\documentclass{article}

\usepackage{amsmath,amssymb} %for math
\usepackage{mathrsfs} %for \mathscr
\usepackage{bm} %for bold letters \bm as vectors 

\usepackage{hyperref}
\usepackage{authblk}

\usepackage{graphicx} %for rotation of symbols, inserting graphics
\usepackage{tikz-cd} % for diagrams
\usepackage{mathtools}  %for \coloneqq, \vcentcolon \mathllap
\usepackage{braket}

\allowdisplaybreaks[2]

\usetikzlibrary{positioning,intersections,calc,decorations.markings,arrows.meta}
\def\sep{1cm}
\def\loopdiam{0.3cm}
\def\unode{0.05cm}
\tikzset{inner sep=0, minimum size=0.5cm}
\tikzset{
block/.style={
  draw, 
  rectangle, 
  minimum height=0.5cm, 
  minimum width=\sep, align=center
  }
}
\tikzset{
	->-/.style={decoration={markings,mark=at position 0.5 with
     {\arrow[xshift=2pt]{Latex[length=4pt,#1]}}},postaction={decorate}},
	-<-/.style={decoration={markings,mark=at position 0.5 with
     {\arrow[xshift=-2pt,rotate=180]{Latex[length=4pt,#1]}}},postaction={decorate}}
}

\usepackage{amsthm} %for thm, def, etc
\theoremstyle{definition} %upshape letters in theorem environment
\newtheorem{dfn}{Definition}[section]
\newtheorem*{dfn*}{Definition}
\newtheorem{rmk}[dfn]{Remark}
\newtheorem{ex}[dfn]{Example}
\newtheorem*{notation*}{Notation}
\newtheorem*{open*}{Open Problems}

\theoremstyle{plain}
\newtheorem{lem}[dfn]{Lemma}
\newtheorem*{lem*}{Lemma}
\newtheorem{prop}[dfn]{Proposition}
\newtheorem*{prop*}{Proposition}
\newtheorem{thm}[dfn]{Theorem}
\newtheorem*{thm*}{Theorem}
\newtheorem{cor}[dfn]{Corollary}
\newtheorem*{cor*}{Corollary}

\newtheorem*{conj*}{Conjecture}

\mathtoolsset{showonlyrefs=true}

\numberwithin{equation}{section}
\newenvironment{sd}{\begin{array}{c} \begin{tikzpicture}}{\end{tikzpicture} \end{array}}

\newcommand{\N}{\mathbb{N}}
\newcommand{\Z}{\mathbb{Z}}

\newcommand{\R}{\mathbb{R}}
\newcommand{\C}{\mathbb{C}}
\newcommand{\T}{\mathbb{T}}

\renewcommand{\tilde}{\widetilde}
\newcommand{\ol}[1]{\overline{#1}}
\newcommand{\abs}[1]{{\left\lvert{#1}\right\rvert}}
\newcommand{\norm}[1]{{\left\lVert{#1}\right\rVert}}
\mathchardef\hyphen="2D
\newcommand{\quotient}[2]{     
\mathchoice{  \text{\raise1ex\hbox{$#1$}\!\Big/\!\lower1ex\hbox{$#2$}} }% \displaystyle
                  {  {#1}\,/\,{#2}  }% \textstyle
                  {  {#1}\,/\,{#2}  }% \scriptstyle
                  {  {#1}\,/\,{#2}  }% \scriptscriptstyle
}

\newcommand{\id}{\mathrm{id}}

\newcommand{\Spec}{\mathop{\mathrm{spec}}\nolimits}

\newcommand{\ran}{\mathop{\mathrm{range}}\nolimits}
\newcommand{\Tr}{\mathop{\mathrm{Tr}}\nolimits}

\newcommand{\ad}{\mathop{\mathrm{ad}}\nolimits}

\newcommand{\op}{\mathrm{op}}

\title{Algebraic connectedness and bipartiteness of quantum graphs}
\author{Junichiro Matsuda \thanks{
Department of Mathematics, Kyoto University, Kitashirakawa Oiwakecho, Sakyo-ku, Kyoto, Japan 606-8502.
Email: {\tt j.matsuda@math.kyoto-u.ac.jp}}}
\affil{Kyoto University
}
\date{}

\begin{document}

\maketitle

\begin{abstract}
Connectedness and bipartiteness are basic properties of classical graphs, 
and the purpose of this paper is to investigate the case of quantum graphs.
We introduce the notion of connectedness and bipartiteness of quantum graphs in terms of graph homomorphisms. This paper shows that regular tracial quantum graphs have the same algebraic characterization of connectedness and bipartiteness as classical graphs.
We also prove the equivalence between bipartiteness and two-colorability of quantum graphs by comparing two notions of graph homomorphisms respecting adjacency matrices or edge spaces.
In particular, all kinds of quantum two-colorability are mutually equivalent for regular connected tracial quantum graphs.
\end{abstract}

The quantum graphs are a non-commutative analogue of classical graphs and recently developed in the interactions between theories of operator algebras, quantum information, non-commutative geometry, quantum groups, etc.

Since quantum graphs as adjacency matrices were introduced by Musto, Reutter, Verdon \cite{Musto2018compositional}, 
there has been substantial activity towards clarifying the relation between the property of a quantum graph and the spectrum of the adjacency matrix.
It is classically known that the spectrum of the adjacency matrix can characterize some properties of a (regular) classical graph. Hoffman \cite{Hoffman1963polynomial} showed that a connected $d$-regular graph is bipartite if and only if $-d$ is an eigenvalue of the adjacency matrix, and it was already known in Fiedler \cite{Fiedler1973algebraic} that the connectedness of a graph is equivalent to the nonzero spectral gap of the graph Laplacian (cf.~\cite{Chung1997}). 
So it is natural to expect that quantum graphs have similar spectral characterizations, and indeed Ganesan \cite{Ganesan2023spectral} shows that such a spectral approach is valid for the chromatic numbers of quantum graphs.

Similarly to the classical case, the degree of a regular quantum graph is shown to be the spectral radius of the adjacency matrix. Thus it makes sense to consider the behavior of the spectrum in $[-d,d]$ for $d$-regular undirected quantum graphs.
In this paper, we introduce bipartiteness and connectedness for quantum graphs in terms of graph homomorphisms, and we give their spectral characterizations for regular quantum graphs.

Regarding the notion of graph homomorphisms, we compare two notions of graph homomorphisms, one is defined in this paper and compatible with adjacency matrices, and the other is defined in \cite{BGH2022quantum} 
and compatible with edge spaces.
We prove that these two notions are equivalent particularly in the case of quantum-to-classical graph homomorphism,
that is, any edge is mapped to the edges if the adjacency matrix is mapped to edges.
In the proof, string diagrams (c.f.~\cite{Vicary2011categorical, Musto2018compositional, Matsuda2022classification}) play a significant role
to deduce positivity from the symmetry of the diagram. 
As its corollary, we obtained that the local two-colorability is equivalent to bipartiteness for tracial real quantum graphs. 
Moreover, combining the results in this paper, it follows that 
all kinds of quantum two-colorability are mutually equivalent
for connected regular undirected tracial quantum graphs.

This paper is organized as follows.

In section 1, we prepare the basic terminology of quantum graphs and string diagrams referring to the preceding research and show some lemmas for later use.

In section 2, we introduce the graph gradient to show the positivity of graph Laplacian. From the positivity, we deduce the spectral bound by the degree of regular real quantum graphs. On the way, we show that quantum graphs do not admit an orientation in general.
\begin{thm*}[Proposition \ref{prop:specrad=d}]
Let $\mathcal{G}=(B,\psi,A)$ be a $d$-regular real quantum graph. The spectral radius $r(A)$ of the adjacency matrix satisfies $r(A)=d$.
\end{thm*}
\begin{thm*}[{Theorem \ref{thm:norm=deg}}]
Let $\mathcal{G}=(B,\psi,A)$ be a $d$-regular quantum graph. Then the identity of the operator norm on $B(L^2(\mathcal{G}))$ and the degree
\[
\norm{A}_\op=d
\]
 holds if either of the following is satisfied:
\begin{description}
\item[$(1)$]
$\mathcal{G}$ is undirected, whence $\Spec(A)\subset [-d,d]$;
\item[$(2)$]
both $A$ and $A^\dagger$ are real;
\item[$(3)$]
$\mathcal{G}$ is real and tracial.
\end{description}
\end{thm*}

In section 3, we introduce our notion of graph homomorphism, connectedness, and bipartiteness and prove their algebraic characterizations by the spectrum of the adjacency matrix.
In the proof, Lemma \ref{lem:posnegdec} plays an essential role in controlling the decomposition of a self-adjoint operator into a subtraction of positive elements.
\begin{thm*}[Theorem \ref{thm:conn=algconn}, Theorem \ref{thm:connbipartite}, Theorem \ref{thm:bipartite}]
Let $\mathcal{G} = (B,\psi,A)$ be a $d$-regular undirected tracial quantum graph.
\begin{itemize}
\item 
$\mathcal{G}$ is connected if and only if
$d \in \Spec(A)$ is a simple root.
\item
$\mathcal{G}$ has a bipartite component if and only if
$-d \in \Spec(A)$.
If $d=0$, we require ${\dim B} \geq 2$.
\end{itemize}
If moreover $\mathcal{G}$ is connected, then
\begin{itemize}
\item
$\mathcal{G}$ is bipartite if and only if
$-d \in \Spec(A)$.
If $d=0$, we require ${\dim B} \geq 2$.
\end{itemize}
\end{thm*}

In section 4, we give a modified generalization of $t$-homomorphisms ($t \in \{loc,q,qa,qc,C^*,alg\}$) introduced by \cite{BGH2022quantum} in the quantum-to-classical cases
and show that it agrees with the quantum homomorphism defined by \cite{Musto2018compositional} in terms of string diagrams.
Then we prove that our graph homomorphisms and $loc$-homomorphisms coincide under some assumptions.
\begin{thm*}[Theorem \ref{thm:hom-lochomequiv_tracial}]
Let $\mathcal{G}_j$ for $j=0,1$ 
be real tracial quantum graphs such that $\mathcal{G}_1$ is Schur central.
Then $f^\op:\mathcal{G}_0 \to \mathcal{G}_1$ is a graph homomorphism
if and only if $(f,\C):\mathcal{G}_0 \overset{}{\to} \mathcal{G}_1$ is a $loc$-homomorphism.
\end{thm*}
This result yields the equivalence of bipartiteness and local two-colorability,
which implies the equivalence of all the $t$-$2$ colorability.
\begin{thm*}[Theorem \ref{thm:bipartiff2col}]
Let $\mathcal{G}$ be a real tracial quantum graph. Then $\mathcal{G}$ is bipartite if and only if it is $loc$-$2$ colorable.
\end{thm*}
\begin{thm*}[Corollary \ref{cor:2colequiv}]
Let $\mathcal{G}=(B,\psi,A)$ be a connected $d$-regular undirected tracial quantum graph.
The following are equivalent:
\begin{description}
\item[(1)] $\mathcal{G}$ is $loc$-$2$ colorable;
\item[(2)] $\mathcal{G}$ is $alg$-$2$ colorable;
\item[(3)] $\mathcal{G}$ has a symmetric spectrum;
\item[(4)] $-d \in \Spec(A)$.
If $d=0$, we require ${\dim B} \geq 2$;
\item[(5)] $\mathcal{G}$ is bipartite.
\end{description}
\end{thm*}

Our main results are restricted to regular tracial quantum graphs.
So the non-tracial versions and the equivalence of the spectral gap of the graph Laplacian and the connectedness of irregular quantum graphs are left open.
The relation between our connectedness and the operator space theoretic connectedness \cite{Chavez2021connectivity} of quantum graphs is also left open.

\section*{Acknowledgement}

I am very grateful to Dr.~Priyanga Ganesan for informing me of the status of her parallel discussions with Professor Kristin Courtney about the algebraic characterization of connectedness.
I am very grateful to Professors Michael Brannan and Matthew Kennedy and their students Ms.~Jennifer Zhu and Ms.~Larissa Kroell for multiple discussions on the occasion of my stay at the University of Waterloo, and for connecting me to Dr.~Ganesan.
I would like to show my greatest gratitude to my supervisor Professor Beno\^{i}t Collins who guided me with generosity and patience.
I would like to thank my colleague Mr.~Akihiro Miyagawa for having mathematical discussions with me and pointing out typos.

This work was supported by JSPS KAKENHI Grant Number JP23KJ1270 and JST, the establishment of university fellowships towards the creation of science technology innovation, Grant Number JPMJFS2123.

\section{Preliminaries}

We denote by $(\cdot)^*$ the involution on a $*$-algebra and 
by $(\cdot)^\dagger$ the adjoint of an operator between Hilbert spaces. 

For a state $\psi$ on a $C^*$-algebra $B$, 
we denote the GNS space by $L^2(B,\psi)$, 
which is the Hausdorff completion of $B$ with respect to the sesquilinear form 
$\braket{x|y}=\braket{x|y}_\psi=\psi(x^* y)$ for $x,y \in B$. 
If $\dim B$ is finite and $\psi$ is faithful, 
we identify $B \ni x = \ket{x} \in L^2(B,\psi)$. 
Then we have the multiplication $m:B \otimes B \ni x \otimes y \mapsto xy \in B$ and the comultiplication $m^\dagger: B \to B \otimes B$. 
The unit $1_B$ is identified with a map $\C \ni 1 \mapsto 1_B \in B$ and its adjoint is the counit $\psi =1_B^\dagger =\bra{1_B}$.

Via a multimatrix presentation $B=\bigoplus_s M_{n_s}$, we denote the direct sum of unnormalized traces by $\Tr=\bigoplus_s \Tr_{n_s}$ (independent of the choice of the presentation) and denote the density matrix of 
$\psi=\Tr(Q \, \cdot)$ by $Q \in B$. If $\psi$ is faithful, $Q$ is positive and invertible.

\subsection{String diagrams}

We denote linear operators by string diagrams, which encode the compositions of operators from the bottom to the top. See 
\cite{Vicary2011categorical, Musto2018compositional} 
for string diagrams of Frobenius algebras and tracial quantum graphs and \cite{Matsuda2022classification} for diagrams of non-tracial quantum graphs.

In particular, we put 
\begin{align}
1_B=\ket{1_B} =
\begin{sd}
\draw (0,0)--++(0,-\sep/4) arc(90:-270:\unode);
\end{sd};
m=\begin{sd}
\draw (0,0) to[out=90,in=90] coordinate[midway](m) (\sep/2,0) ;
\draw (m)--++(0,\sep/4);
\end{sd};
\psi=\bra{1_B}=
\begin{sd}
\draw (0,0)--++(0,\sep/4) arc(-90:270:\unode);
\end{sd};
m^\dagger=\begin{sd}
\draw (0,0) to[out=-90,in=-90] coordinate[midway](m) (\sep/2,0) ;
\draw (m)--++(0,-\sep/4);
\end{sd}.
\end{align}
We abbreviate
$\psi m=\begin{sd}
\draw (0,0) to[out=90,in=90] coordinate[midway](m) (\sep/2,0) ;
\draw (m)--++(0,\sep/8) arc(-90:270:\unode);
\end{sd}
=\begin{sd}
\draw (0,0) to[out=90,in=90] coordinate[midway](m) (\sep/2,0) ;
\end{sd},
m^\dagger 1=\begin{sd}
\draw (0,0) to[out=-90,in=-90] coordinate[midway](m) (\sep/2,0) ;
\draw (m)--++(0,-\sep/8) arc(90:-270:\unode);
\end{sd}
=\begin{sd}
\draw (0,0) to[out=-90,in=-90] coordinate[midway](m) (\sep/2,0) ;
\end{sd}.$

The key point is that the string diagrams allow us graphical calculation (continuous deformation of diagrams) via associativity
$\begin{sd}
\draw (0,0)to[out=90,in=90]coordinate[midway](m1) (-\sep/4,0) 
	(m1)to[out=90,in=90] coordinate[midway](m2) (\sep/4,0) 
	(m2)--++(0,\sep/8);
\end{sd}
=
\begin{sd}
\draw (0,0)to[out=90,in=90]coordinate[midway](m1) (\sep/4,0) 
	(m1)to[out=90,in=90] coordinate[midway](m2) (-\sep/4,0) 
	(m2)--++(0,\sep/8);
\end{sd}$, 
coassociativity
$\begin{sd}
\draw (0,0)to[out=-90,in=-90]coordinate[midway](m1) (-\sep/4,0) 
	(m1)to[out=-90,in=-90] coordinate[midway](m2) (\sep/4,0) 
	(m2)--++(0,-\sep/8);
\end{sd}
=
\begin{sd}
\draw (0,0)to[out=-90,in=-90]coordinate[midway](m1) (\sep/4,0) 
	(m1)to[out=-90,in=-90] coordinate[midway](m2) (-\sep/4,0) 
	(m2)--++(0,-\sep/8);
\end{sd}$, 
and the Frobenius equality
$\begin{sd}
\draw (0,0)to[out=90,in=90]coordinate[midway](m1) (\sep/4,0) 
		--++(0,-\sep/6) 
	(0,0)to[out=-90,in=-90] coordinate[midway](m2) (-\sep/4,0) 
		--++(0,\sep/6) 
	(m1)--++(0,\sep/8) (m2)--++(0,-\sep/8) ;
\end{sd}
=
\begin{sd}
\draw (0,0)to[out=90,in=90]coordinate[midway](m1) (-\sep/3,0) 
	(0,\sep/3)to[out=-90,in=-90] coordinate[midway](m2) (-\sep/3,\sep/3) 
	(m2)--(m1);
\end{sd}
=
\begin{sd}
\draw (0,0)to[out=90,in=90]coordinate[midway](m1) (-\sep/4,0) 
		--++(0,-\sep/6) 
	(0,0)to[out=-90,in=-90] coordinate[midway](m2) (\sep/4,0) 
		--++(0,\sep/6) 
	(m1)--++(0,\sep/8) (m2)--++(0,-\sep/8) ;
\end{sd}$.
If $\psi$ is non-tracial, note that we sometimes need to deal with
\[
\begin{sd}
	\coordinate (B1) ;
	\coordinate[above =1/2 of B1] (B1a) ;
	\draw (B1) to[out=90,in=90] ($(B1)!1/2!(B1a)+(-\loopdiam,0)$)
	to[out=-90,in=-90] (B1a);
\end{sd}
\coloneqq
\begin{sd}
	\coordinate (B1) ;
	\draw (B1)--++(0,1/3) to[out=90,in=-90] ++(-1/3,1/3)  
		to[out=90,in=90]  ++(-1/3,0)--++(0,-1/3)
		to[out=-90,in=-90]  ++(1/3,0)
	to[out=90,in=-90] ++(1/3,1/3)--++(0,1/3);
\end{sd}
=Q^{-1}(\cdot)Q=\sigma_i,
\quad
\begin{sd}
	\coordinate (B1) ;
	\coordinate[above =1/2 of B1] (B1a) ;
	\draw (B1) to[out=90,in=90] ($(B1)!1/2!(B1a)+(\loopdiam,0)$)
	to[out=-90,in=-90] (B1a);
\end{sd}=Q(\cdot)Q^{-1}=\sigma_{-i},
\]
where $\sigma_z:B \to B$ for $z \in \C$ are the modular automorphisms 
$\sigma_z(x)=Q^{iz} x Q^{-iz}$ for the positive invertible density $Q \in B$ of the faithful state $\psi=\Tr(Q\, \cdot)$.

\subsection{Quantum graphs}

\begin{dfn}[{\cite{Musto2018compositional, Brannan2019bigalois}}]
A \emph{quantum set} is $(B,\psi)$ consisting of a finite-dimensional $C^*$-algebra $B$ with a $\delta$-form $\psi:B \to \C$, where the $\delta$-form is defined as a faithful state satisfying $mm^\dagger=\delta^2 \id_B$ for $\delta \geq 0$. 
\end{dfn}

We denote by $\tau_B$ the unique tracial $\delta=\sqrt{\dim B}$-form on $B$. 
If $B=\bigoplus_s M_{n_s}$ and $\psi=\Tr(Q\, \cdot)$ with $Q=\bigoplus Q_s$, then $\psi$ is a $\delta$-form if and only if $\Tr(Q_s^{-1})=\delta^2$ for all $s$, and $\tau_B= \bigoplus_s n_s \Tr_s /\dim B$.

\begin{dfn}[{\cite{Musto2018compositional, Brannan2019bigalois, Brannan2022quantum}}]
Let $(B,\psi)$ be a quantum set. We define the \emph{Schur product} $S \bullet T$ and  the involution $T^*$ of 
$S,T \in B(L^2(B,\psi))$ by
\begin{align}
S \bullet T \coloneqq \delta^{-2} m(S \otimes T)m^\dagger 
= \delta^{-2} 
\begin{sd}
	\coordinate (H1) at (0,-3/5*\sep) ;
	\coordinate (H3) at (0,3/5*\sep);
	\path (-\sep/3,0)node[circle,draw](A1){$S$} 
		(\sep/3,0)node[circle,draw](A2){$T$};
	\draw (A2.south) to[out=-90,in=-90] coordinate[midway] (m2) (A1.south) ;
	\draw (A2.north) to[out=90,in=90] coordinate[midway] (m3) (A1.north) ;
	\draw (m2)--(H1) (m3)--(H3);
\end{sd}; 
\quad
T^* \coloneqq (T(\cdot)^*)^* = 
\begin{sd}
	\node[circle,draw] (G) {$T^\dagger$};
	\draw (G)to[out=-90,in=-90] coordinate[midway] (m*) 
		([xshift=\sep/2]G.south);
	\draw (G)to[out=90,in=90] coordinate[midway] (m) 
		([xshift=-\sep/2]G.north);
	\draw 
		([xshift=\sep/2]G.south)--++(0,\sep/2) 
		([xshift=-\sep/2]G.north)--++(0,-\sep/2);
\end{sd},		
\end{align} 
with which $B(L^2(B,\psi))$ forms a $*$-algebra isomorphic to $B^\op\otimes B$. 
See for example \cite[Lemma 2.13]{Matsuda2022classification} about the identity of the involution and the diagram.
 The correspondence is given by 
\[
B(L^2(B,\psi)) \ni T \leftrightarrow
p_T \coloneqq
\delta^{-2} \begin{sd}
	\coordinate (H1) at (0,-3/5*\sep) ;
	\coordinate (H3) at (0,3/5*\sep);
	\path (-\sep/3,0)node[draw](A1){$\sigma_{i/2}$} 
		(\sep/3,0)node[draw](A2){$T$};
	\draw (A2.south) to[out=-90,in=-90] coordinate[midway] (m2) (A1.south) ;
	\draw (A2.north)--++(0,\sep/4) (A1.north)--++(0,\sep/4) ;
\end{sd}
\in B^\op \otimes B
\]
where $\sigma_{i/2}=Q^{-1/2}(\cdot)Q^{1/2}:B \to B$ is a modular automorphism and $p_T= \Psi_{0,1/2}^\prime(T)$ defined in \cite[Definition 5.1]{Daws2022quantum}.
See \cite{Gromada2022some} for the tracial setting and \cite{Daws2022quantum, Wasilewski2023quantum} for the details in general setting.

We say that $T:B \to B$ is \emph{real} if $T^*=T$ (i.e., $*$-preserving. 
cf.~\cite{Musto2018compositional});
 $T$ is \emph{Schur idempotent} if $T \bullet T=T$;
and $T$ is a \emph{Schur projection} if it is real and Schur idempotent.
\end{dfn}

We often use the realness of $T:B \to B$ in the form of
\begin{align}
\begin{sd}
\node[circle,draw](A){$T^\dagger$};
\draw (A)to[out=-90,in=-90] coordinate[midway](m) 
		([xshift=\sep/2]A.south)
		--([xshift=\sep/2]A.north)--++(0,\sep/4)
		(A.north)--++(0,\sep/4);
\end{sd}
=
\begin{sd}
\node[circle,draw](A){$T$};
\draw (A)to[out=-90,in=-90] coordinate[midway](m) 
		([xshift=-\sep/2]A.south)
		--([xshift=-\sep/2]A.north)--++(0,\sep/4)
		(A.north)--++(0,\sep/4);
\end{sd}
\textrm{ or }
\begin{sd}
\node[circle,draw](A){$T$};
\draw (A)to[out=90,in=90] coordinate[midway](m) 
		([xshift=\sep/2]A.north)
		--([xshift=\sep/2]A.south)--++(0,-\sep/4)
		(A.south)--++(0,-\sep/4);
\end{sd}
=
\begin{sd}
\node[circle,draw](A){$T^\dagger$};
\draw (A)to[out=90,in=90] coordinate[midway](m) 
		([xshift=-\sep/2]A.north)
		--([xshift=-\sep/2]A.south)--++(0,-\sep/4)
		(A.south)--++(0,-\sep/4);
\end{sd}. 	\label{realness}
\end{align}
Note that if $T$ is real, then $T^\dagger$ is real if and only if $T$ commutes with modular automorphisms $\sigma_z$ (c.f.~\cite[Proposition 2.15]{Matsuda2022classification}, \cite[Lemma 2.1]{Wasilewski2023quantum}).
This means that we cannot always replace $T$ with $T^\dagger$ in \eqref{realness}.

\begin{dfn}[KMS adjoint]
Wasilewski \cite{Wasilewski2023quantum} pointed out that the \emph{KMS inner product} $\braket{x|y}=\psi(x^* \sigma_{-i/2}(y))$ on $B$
behaves better than the GNS inner product 
$\braket{x|y}_\psi=\psi(x^* y)$ when we define non-tracial quantum Cayley graphs. 
They coincide if $\psi$ is tracial. 
The \emph{KMS adjoint} is the adjoint of an operator on (tensor powers of) $B$ with respect to the KMS inner product. 

The relation between the GNS adjoint $T^\dagger$ and the KMS adjoint $T^\ddagger$ of $T:B^{\otimes m}\to B^{\otimes n}$ is given by
$T^\ddagger = \sigma_{i/2}^{\otimes m} T^\dagger \sigma_{-i/2}^{\otimes n}$.
Define 
$\begin{sd}
	\coordinate (B1) ;
	\coordinate[above =1/2 of B1] (B1a) ;
	\draw (B1) to[out=90,in=0] ($(B1)!1/2!(B1a)+(-\loopdiam,0)$)
	to[out=0,in=-90] (B1a);
\end{sd} \coloneqq \sigma_{i/2}$,
$\begin{sd}
	\coordinate (B1) ;
	\coordinate[above =1/2 of B1] (B1a) ;
	\draw (B1) to[out=90,in=180] ($(B1)!1/2!(B1a)+(\loopdiam,0)$)
	to[out=180,in=-90] (B1a);
\end{sd} \coloneqq \sigma_{-i/2}$, and
$\begin{sd}
\draw (0,0) to[out=90,in=-90] (\sep/4,\sep/2) 
	 to[out=-90,in=90] (\sep/2,0) ;
\end{sd}
\coloneqq
\begin{sd}
	\coordinate (B1) ;
	\coordinate[above =1/4 of B1] (B1a) ;
	\coordinate[left=1/2 of B1] (B2) ;
	\coordinate[above =1/4 of B2] (B2a) ;
	\draw (B1) to[out=90,in=180] ($(B1)!1/2!(B1a)+(\loopdiam/2,0)$)
	to[out=180,in=-90] (B1a) to[out=90,in=90] (B2a)--(B2);
\end{sd}
=
\begin{sd}
	\coordinate (B1) ;
	\coordinate[above =1/4 of B1] (B1a) ;
	\coordinate[left=1/2 of B1] (B2) ;
	\coordinate[above =1/4 of B2] (B2a) ;
	\draw (B2) to[out=90,in=0] ($(B2)!1/2!(B2a)+(-\loopdiam/2,0)$)
	to[out=0,in=-90] (B2a) to[out=90,in=90] (B1a)--(B1);
\end{sd}$,
where the cusp stands for the operator in the middle of 
the straight string $\id_B$ and the loops $\sigma_{\pm i}$.
Then the KMS inner product is drawn as
$\braket{x|y}=\begin{sd}
\path (0,0) node[circle,draw,fill=white] (B1){$x^*$}
	 (\sep*0.6,0) node[circle,draw,fill=white] (B2){$y$};
\draw (B1) to[out=90,in=-90] ($(B1)!0.5!(B2)+(0,\sep/2)$) 
	 to[out=-90,in=90] (B2) ;
\end{sd}$ and the relation of the KMS adjoint and the involution is
$T^*=
\begin{sd}
	\node[circle,draw] (G) {$T^\ddagger$};
	\draw (G) to[out=-90,in=90] 
		($(G.south)!1/2!([xshift=\sep/2]G.south)+(0,-\loopdiam/2)$)
		to[out=90,in=-90]
		([xshift=\sep/2]G.south);
	\draw (G) to[out=90,in=-90]
		 ($(G.north)!1/2!([xshift=-\sep/2]G.north)+(0,\loopdiam/2)$)
		to[out=-90,in=90]
		([xshift=-\sep/2]G.north);
	\draw 
		([xshift=\sep/2]G.south)--++(0,\sep/2) 
		([xshift=-\sep/2]G.north)--++(0,-\sep/2);
\end{sd}$. Thus in terms of the KMS adjoint, the realness \eqref{realness} of $T$ is replaced by
\[
\begin{sd}
\node[circle,draw](A){$T^\ddagger$};
\draw (A)to[out=-90,in=90] 
		($(A.south)!1/2!([xshift=\sep/2]A.south)+(0,-\loopdiam/2)$)
		to[out=90,in=-90]
		([xshift=\sep/2]A.south)
		--([xshift=\sep/2]A.north)--++(0,\sep/4)
		(A.north)--++(0,\sep/4);
\end{sd}
=
\begin{sd}
\node[circle,draw](A){$T$};
\draw (A)to[out=-90,in=90] 
		($(A.south)!1/2!([xshift=-\sep/2]A.south)+(0,-\loopdiam/2)$)
		to[out=90,in=-90]
		([xshift=-\sep/2]A.south)
		--([xshift=-\sep/2]A.north)--++(0,\sep/4)
		(A.north)--++(0,\sep/4);
\end{sd}
\textrm{ or }
\begin{sd}
\node[circle,draw](A){$T$};
\draw (A)to[out=90,in=-90]
		 ($(A.north)!1/2!([xshift=\sep/2]A.north)+(0,\loopdiam/2)$)
		to[out=-90,in=90]
		([xshift=\sep/2]A.north)
		--([xshift=\sep/2]A.south)--++(0,-\sep/4)
		(A.south)--++(0,-\sep/4);
\end{sd}
=
\begin{sd}
\node[circle,draw](A){$T^\ddagger$};
\draw (A)to[out=90,in=-90]
		 ($(A.north)!1/2!([xshift=-\sep/2]A.north)+(0,\loopdiam/2)$)
		to[out=-90,in=90]
		([xshift=-\sep/2]A.north)
		--([xshift=-\sep/2]A.south)--++(0,-\sep/4)
		(A.south)--++(0,-\sep/4);
\end{sd}.
\]
A benefit of KMS adjoint is that $T^\ddagger$ is real if and only if $T$ is real. Indeed, the flip invariance
$\begin{sd}
\draw (0,0) to[out=90,in=-90] (\sep*3/8,\sep/4) 
	 to[out=90,in=-90] (\sep/4,\sep/2) 
	to[out=-90,in=90] (\sep/8,\sep/4)
	 to[out=-90,in=90] (\sep/2,0) ;
\end{sd}
=
\begin{sd}
	\coordinate (B1) ;
	\coordinate (B1a) at (-1/6,1/4);
	\coordinate (B2) at (-1/2,0) ;
	\coordinate (B2a) at (-1/3, 1/4);
	\draw (B1) to[out=90,in=0] ($(B1)!1/2!(B1a)+(-\loopdiam/2,0)$)
	to[out=0,in=-45] (B1a) 
	to[out=135,in=180]($(B1a)!1/2!(B2a)+(0,\loopdiam)$)
	to[out=0,in=45] (B2a)to[out=-135,in=90](B2);
\end{sd}
=
\begin{sd}
\draw (0,0) to[out=90,in=-90] (\sep/4,\sep/2) 
	 to[out=-90,in=90] (\sep/2,0) ;
\end{sd}$
implies the equivalence between the realness of $T$ and that of $T^\ddagger$ by flipping the strings:
\begin{align}
\begin{sd}
\node[circle,draw](A){$T$};
\draw (A)to[out=90,in=-90]
		 ($(A.north)!1/2!([xshift=\sep/2]A.north)+(0,\loopdiam/2)$)
		to[out=-90,in=90]
		([xshift=\sep/2]A.north)
		--([xshift=\sep/2]A.south)--++(0,-\sep/4)
		(A.south)--++(0,-\sep/4);
\end{sd}
=
\begin{sd}
\node[circle,draw](A){$T^\ddagger$};
\draw (A)to[out=90,in=-90]
		 ($(A.north)!1/2!([xshift=-\sep/2]A.north)+(0,\loopdiam/2)$)
		to[out=-90,in=90]
		([xshift=-\sep/2]A.north)
		--([xshift=-\sep/2]A.south)--++(0,-\sep/4)
		(A.south)--++(0,-\sep/4);
\end{sd}
\iff
\begin{sd}
\node[circle,draw](A){$T$};
\draw (A)to[out=90,in=-90]
		 ($(A.north)!1/2!([xshift=-\sep/2]A.north)+(0,\loopdiam/2)$)
		to[out=-90,in=90]
		([xshift=-\sep/2]A.north)
		--([xshift=-\sep/2]A.south)--++(0,-\sep/4)
		(A.south)--++(0,-\sep/4);
\end{sd}
=
\begin{sd}
\node[circle,draw](A){$T^\ddagger$};
\draw (A)to[out=90,in=-90]
		 ($(A.north)!1/2!([xshift=\sep/2]A.north)+(0,\loopdiam/2)$)
		to[out=-90,in=90]
		([xshift=\sep/2]A.north)
		--([xshift=\sep/2]A.south)--++(0,-\sep/4)
		(A.south)--++(0,-\sep/4);
\end{sd}.	\label{Treal=Tddreal}
\end{align}
Since the GNS adjoint is easier to treat in string diagrams, we stick to the GNS inner product in this paper.
\end{dfn}

\begin{dfn}[\cite{Musto2018compositional,Brannan2019bigalois}]
A quantum graph is a triple $\mathcal{G}=(B,\psi,A)$ consisting of a quantum set $(B,\psi)$ and an operator $A: B \to B$ satisfying Schur idempotence $A \bullet A=A$, which is called the adjacency matrix.
\end{dfn}

We denote the GNS space of a quantum graph $\mathcal{G}=(B,\psi,A)$ by $L^2(\mathcal{G}) \coloneqq L^2(B,\psi)$.

\begin{dfn}
Let $\mathcal{G}=(B,\psi,A)$ be a quantum graph.
\begin{itemize}
\item[\cite{Brannan2019bigalois}]
$\mathcal{G}$ is tracial (or symmetric) if $\psi$ is tracial, i.e., $\psi=\tau_B$;
\item[\cite{Musto2018compositional}]
$\mathcal{G}$ is real if $A$ is real $A^*=A$. The realness is equivalent to the complete positivity by the Schur idempotence of $A$ (cf.~\cite[Proposition 2.23]{Matsuda2022classification}, \cite[Remark 3.2]{Wasilewski2023quantum});
\item[\cite{Musto2018compositional}]
$\mathcal{G}$ is undirected if $A$ is both real and self-adjoint. 
This is equivalent to GNS symmetry (c.f.~\cite{Wasilewski2023quantum}) $\psi((Ax)y)=\psi(x(Ay))$ under the realness by \cite[Lemma 2.22]{Matsuda2022classification};
\item[\cite{Wasilewski2023quantum}]
$\mathcal{G}$ is KMS symmetric if $A$ is both real and KMS self-adjoint $A=A^\ddagger$;
\item[\cite{Musto2018compositional}]
$\mathcal{G}$ is reflexive (or has all loops) if $A \bullet \id=\id$;
\item[\cite{Musto2018compositional}]
$\mathcal{G}$ is irreflexive (or has no loops) if $A \bullet \id=0$;
\item[\cite{Gromada2022some}]
$\mathcal{G}$ has no partial loops if $A \bullet \id=\id \bullet A$;
\item[\cite{Matsuda2022classification}]
$\mathcal{G}$ is $d$-regular if $A1_B=d1_B=A^\dagger 1_B$. The $d \in \C$ is the degree of $\mathcal{G}$;
\item
$\mathcal{G}$ is Schur central if $A \bullet \cdot=\cdot \bullet A$, i.e., $A$ is central with respect to the Schur product.
\end{itemize}
\end{dfn}

\begin{lem}
Let $\mathcal{G}=(B,\psi,A)$ be a $d$-regular real quantum graph.
It follows that $d \in \R$.
\end{lem}

\begin{proof}
We have $d=\braket{1_B | A1_B}=\braket{1_B | A^* 1_B}=\braket{1_B | (A 1_B)^*}=\ol{d}$.
\end{proof}

\begin{dfn}[{Weaver \cite{Weaver2012quantum}}]
A quantum relation on a von Neumann algebra $B \subset B(H)$ is a weak*-closed $B'$-$B'$-bimodule $\mathcal{S} \subset B(H)$,
where we regard $B(H)$ as the dual of the trace class $TC(H)$ via the coupling $(S,T)\mapsto \Tr(ST)$.

If we chose $H=L^2(B,\psi)$ for a quantum set $(B,\psi)$, then a quantum relation on $B=\lambda(B) \subset B(H)$ is a $\rho(B)$-$\rho(B)$-bimodule $\mathcal{S} \subset B(H)$ where $\lambda$ (resp. $\rho$) is the left (resp. right) regular representation with respect to $\psi$.
\end{dfn}

Quantum relations on a quantum set $(B,\psi)$ are identified with $B$-$B$-bimodules $\mathcal{S} \subset B \otimes B$ via the identification:
\begin{align}
\iota:B(L^2(B,\psi)) \ni T \mapsto
\iota(T)=\begin{sd}
	\coordinate (H1) at (0,-3/5*\sep) ;
	\coordinate (H3) at (0,3/5*\sep);
	\path (-\sep/3,0)node(A1){} 
		(\sep/3,0)node[draw](A2){$T$};
	\draw (A2.south) to[out=-90,in=-90] coordinate[midway] (m2) (A1.south) ;
	\draw (A2.north)--++(0,\sep/4) (A1.south)--(A1.north)--++(0,\sep/4) ;
\end{sd}
\in B \otimes B.	\label{B(B)=BoxB}
\end{align}
See for example \cite{ManuilovTroitsky2005}, \cite[Appendix F]{Brown2008} about bimodules over von Neumann algebras, and \cite{Musto2018compositional} about the one-to-one correspondence above.

The linear isomorphism $\iota$ is the linear extension of $\iota(\ket{x}\bra{y})=\sigma_{-i}(y^*) \otimes x$ for $x,y \in B$.
We endow $B(L^2(B,\psi))$ with a Hilbert space structure via $\iota$,
i.e., 
\begin{align}
\braket{S|T}=\braket{\iota(S)|\iota(T)}_{\psi\otimes\psi}
=\Psi(S^\dagger T)=\delta^2 \braket{1|S^* \bullet T|1},
	\label{innerprodonB(L2)}
\end{align}
where $\Psi=\delta^2  \braket{1|\id_B \bullet \cdot |1} =
\begin{sd}
\node[dashed,draw, minimum size=0.3cm](T)at(0,0){};
\draw (T.north) to[out=90,in=90] ([xshift=-0.4cm]T) to[out=-90,in=-90] (T.south);
\end{sd}$
is an extension of $\delta^2 \psi$ on $\rho(B)$ to $B(L^2(B,\psi))$.

Let $P:L^2(B,\psi)^{\otimes 2 }\to L^2(B,\psi)^{\otimes 2}$ be the orthogonal projection onto the $B$-$B$-bimodule $\mathcal{S} \subset B \otimes B = L^2(B,\psi)^{\otimes 2}$. Then $P$ is $B$-$B$-bimodule map, i.e., $P(x \xi y)=xP(\xi)y$ for $x,y \in B$ and $\xi \in B \otimes B$.

There is a one-to-one correspondence (cf.~\cite{Musto2018compositional}) between $B$-$B$-bimodule projections $P$ on $B \otimes B$ and real quantum graphs $(B,\psi,A)$ as follows:
\begin{align}
	P=P_A &\coloneqq \delta^{-2}
	\begin{sd}
	\node[circle,draw] (G) {$A$};
	\draw (G)to[out=-90,in=-90] coordinate[midway] (m*) 
		([xshift=-\sep/2]G.south);
	\draw (G)to[out=90,in=90] coordinate[midway] (m) 
		([xshift=\sep/2]G.north);
	\draw (m)--++(0,\sep/4) (m*)--++(0,-\sep/4) 
		([xshift=-\sep/2]G.south)--++(0,\sep/2) 
		([xshift=\sep/2]G.north)--++(0,-\sep/2);
	\end{sd};
	&
	A=A_P &= \delta^2
	\begin{sd}
	\node[block] (P) {$P$};
	\draw ([xshift=\sep/3]P.north)--++(0,1/4);
	\draw ([xshift=-\sep/3]P.north)--++(0,1/4) arc(-90:270:\unode);
	\draw ([xshift=\sep/3]P.south)--++(0,-1/4) arc(90:-270:\unode);
	\draw ([xshift=-\sep/3]P.south)--++(0,-1/4);
	\end{sd}. 	\label{graph-rel}
\end{align}
Note that $P_A=\iota \tilde{P_A} \iota^{-1}$ is the reformulation of left Schur product by $A$: 
\[
\tilde{P_A} = A \bullet (\cdot) :B(L^2(B,\psi)) \ni T \mapsto A\bullet T \in B(L^2(B,\psi)).
\]

\begin{ex}
We denote the reflexive trivial graph on a quantum set $(B,\psi)$ by $T(B,\psi)=(B,\psi,\id_B)$, which is undirected $1$-regular.
This is the quantum version of graphs with all loops. 
Its corresponding quantum relation is the commutant of $B$: 
$\mathcal{S}_{T(B,\psi)}=\lambda(B)'=\rho(B) \subset B(L^2(B,\psi))$.
We abbreviate the classical trivial graphs $T_n=T(\C^n,\tau_{\C^n})$ for $n \in \N$.

We denote the irreflexive complete graph on a quantum set $(B,\psi)$ with $\delta$-form by $K(B,\psi)=(B,\psi, \delta^2 \psi(\cdot)1_B - \id_B)$,
which is undirected $(\delta^2-1)$-regular.
This is the quantum version of graphs with all edges except loops. 
Its corresponding quantum relation is the orthocomplement of the commutant of $B$: 
$\mathcal{S}_{K(B,\psi)}=\mathcal{S}_{T(B,\psi)}^\perp=\rho(B)^\perp \subset B(L^2(B,\psi))$.
We abbreviate the classical complete graphs $K_n=K(\C^n,\tau_{\C^n})$ for $n \in \N$.
We denote by $J=\delta^2 \psi(\cdot)1_B$ the adjacency matrix of the reflexive complete quantum graphs. 
\end{ex}

The degree of a regular classical graph is at most the size of the vertex set.
The value $\delta^2$ plays the role of the size of a quantum set and it bounds the degree.

\begin{lem}\label{lem:deg<dim}
Let $\mathcal{G}=(B,\psi,A)$ be a $d$-regular real quantum graph.
Then $0 \leq d \leq \delta^2$.
In particular, $d=0$ if and only if $A=0$, and $d=\delta^2$ if and only if $A=\delta^2 \psi(\cdot)1_B$.
\end{lem}

\begin{proof}
We use the correspondence between $A$ and a projection $p_A \in B^\op \otimes B$.
Note that the reflexive complete graph $(B,\psi, J=\delta^2 \psi(\cdot)1_B)$ corresponds to the maximal projection $p_J =1 \otimes 1 \in B^\op\otimes B$.
Since $0 \leq p_A \leq p_J$ and $\psi^{\otimes2}$ is a state on $B^\op\otimes B$,
we have 
\[
0 \leq d=\psi(A1_B)=\psi^{\otimes2}(p_A) \leq \psi^{\otimes2}(p_J)=\psi(J1_B)=\delta^2.
\]
Since $\psi^{\otimes 2}$ is faithful,
 $d=0$ holds if and only if $A=0$, 
and $d=\delta^2$ holds if and only if $A=J$.
\end{proof}

Gromada \cite[section 2.3]{Gromada2022some} pointed out that the value 
$\delta^2 \braket{1_B|A|1_B} = \delta^2\psi(A1_B)$ is the number of edges.
This value is strictly positive whenever $A$ is nonzero and real:

\begin{lem}\label{lem:numofedges}
Let $\mathcal{G}=(B,\psi, A)$ be a real quantum graph.
Then $\braket{1_B|A|1_B}\geq 0$ with equality if and only if $A=0$.
\end{lem}

\begin{proof}
Similarly to the proof of Lemma \ref{lem:deg<dim}, we have
\[
\braket{1_B|A|1_B}=\psi^{\otimes2}(p_A)\geq 0.
\]
Since $\psi^{\otimes 2}$ is faithful,
 $\braket{1_B|A|1_B}=0$ holds if and only if $A=0$.
\end{proof}

For later use, we show that the eigenspace for any real eigenvalue of a real quantum graph is spanned by self-adjoint elements:

\begin{lem}\label{lem:saevreal}
Let $(B,\psi,A)$ be a real quantum graph and $x \in B$ be an eigenvector for an eigenvalue $\lambda$ of $A$. Then $x^*$ is an eigenvector for the eigenvalue $\ol{\lambda}$ of $A$.
In particular if $\lambda \in \Spec(A) \cap \R$, then the eigenspace 
$\ker(\lambda \, \id -A)$ and the generalized eigenspace $\ker(\lambda \, \id -A)^{\dim B}$ are spanned by self-adjoint elements.
\end{lem}

\begin{proof}
Taking the involution of $(\lambda \, \id -A)x = 0$, we get 
$(\ol{\lambda}\, \id -A)x^* = (\lambda x)^*  -(Ax)^* 
= ((\lambda\, \id -A)x)^* = 0$. 
If $\lambda$ is real, then both $x$ and $x^*$ are eigenvectors for $\lambda$, hence $\Re x=\frac{x+x^*}{2}, \Im x=\frac{x-x^*}{2i}$ are also eigenvectors for $\lambda$. Since $x$ is arbitrary, $\ker(\lambda\,  \id -A)$ is spanned by self-adjoint elements. Similarly $\ker(\lambda\,  \id -A)^{\dim B}$ is so.
\end{proof}

\section{Spectral bound for regular quantum graphs}

In classical graph theory, $d$-regular graph is known to have spectral radius $d$ (cf.~\cite{Chung1997}) and hence it makes sense to argue whether the second largest eigenvalue is $d$ and the smallest eigenvalue is $-d$.
Here we introduce the notion of graph gradient to prove this spectral bound for regular quantum graphs. 

\subsection{Graph gradient of quantum graphs}

\begin{dfn}
Let $(B,\psi,A)$ be a quantum graph. Define a linear operator 
$\nabla=\nabla_A: B \to B \otimes B$ by
\[
\nabla_A=\delta^{-2}(A^\dagger \otimes \id_B - \id_B \otimes A)m^\dagger
=
\delta^{-2} \left(
\begin{sd}
\node[circle,draw](A){$A^\dagger$};
\draw (A)to[out=-90,in=-90] coordinate[midway](m) 
		([xshift=\sep/2]A.south)
		--([xshift=\sep/2]A.north)--++(0,\sep/4)
		(A.north)--++(0,\sep/4)
		(m)--++(0,-\sep/4);
\end{sd}
-
\begin{sd}
\node[circle,draw](A){$A$};
\draw (A)to[out=-90,in=-90] coordinate[midway](m) 
		([xshift=-\sep/2]A.south)
		--([xshift=-\sep/2]A.north)--++(0,\sep/4)
		(A.north)--++(0,\sep/4)
		(m)--++(0,-\sep/4);
\end{sd}
\right).
\]
We call $\nabla_A$ the graph gradient.
\end{dfn}

This gradient coincides with the classical one in the following manner.

\begin{lem}\label{lem:classicalgrad}
Let $(V,E \subset V \times V)$ be a classical directed graph corresponding to
$(C(V)=\C^n, \tau, A)$ with $A_{ij}=\chi_E (j,i)$ where $\chi_E$ is the indicator function of $E$. The classical graph gradient $\nabla_E:C(V) \to C(E)$ (the so-called coboundary operator in \cite{Chung1997}) is defined by 
\[
\nabla_E f (i,j)=f(j)-f(i)
\qquad f\in C(V), \ \ (i,j)\in E.
\]
It holds that $\nabla_A=\iota \circ \nabla_E$, where 
$\iota: C(E) \to C(V)\otimes C(V)=C(V\times V)$ is the extension of functions on $E$ to $V \times V$ with outside zero.
\end{lem}

\begin{proof}
Note that the evaluation map $C(V) \ni f \mapsto f(i) \in \C$ at $i\in V$ is given by $n \bra{e_i}=n \tau(e_i \ \cdot )$ for the tracial $\sqrt{n}$-form $\tau$ and $m^\dagger e_k=n e_k \otimes e_k$.
By direct computation we have for $f\in C(V)$ and $i,j \in V$ that
\begin{align}
\nabla_A f (i,j) &= n^2 (\bra{e_i} \otimes \bra{e_j}) \nabla_A f
\\
&= n^2 n^{-1} \left( \bra{e_i}A^\dagger \otimes \bra{e_j} - \bra{e_i} \otimes \bra{e_j} A \right)
	n \sum_k f(k) e_k \otimes e_k
\\
&= n^2 (\braket{e_i | A^\dagger f(j) | e_j} n^{-1} - n^{-1} \braket{e_j | A f(i) |e_i})
\\
&= n \sum_{(k,j) \in E} \braket{e_i | e_k} f(j) - n \sum_{(i,k) \in E}\braket{e_j |e_k} f(i)
\\
&= (f(j)-f(i)) \chi_E (i,j) = (\iota \nabla_E f)(i,j).
\end{align}
\end{proof}

The graph gradient $\nabla_A$ is the commutator of the right regular representation $\rho(\cdot)$ and $A$ 
via the identification \eqref{B(B)=BoxB} 
$\iota: B(L^2(\mathcal{G}))\cong B \otimes B$:

\begin{prop}
Let $(B,\psi,A)$ be a real quantum graph. For $x \in B$, we have 
\[
	\delta^{2} \iota^{-1} (\nabla_A x) = [\rho(x),A] 
	\coloneqq \rho(x)A-A\rho(x).
\]
\end{prop} 

\begin{proof}
By direct computation, we get
\[
	\delta^{2} \iota^{-1} (\nabla_A x) 
=
	\begin{sd}
	\node[circle,draw] (G) {$A^\dagger$};
	\draw (G)to[out=90,in=90] coordinate[midway] (m*) 
		([xshift=-\sep/2]G.north);
	\draw (G)to[out=-90,in=-90] coordinate[midway] (m) 
		([xshift=\sep/2]G.south);
	\draw (m)--++(0,-\sep/8)
		node[circle,draw,below] (x){$x$}
		([xshift=-\sep/2]G.north)--++(0,-\sep/2) 
		([xshift=\sep/2]G.south)--++(0,\sep/2);
	\end{sd}
-
	\begin{sd}
	\node[circle,draw](A){$A$};
	\draw (A)to[out=-90,in=-90] coordinate[midway](m) 
		([xshift=-\sep/2]A.south)
		to[out=90,in=90]([xshift=-\sep]A.south)--++(0,-\sep/4)
		(A.north)--++(0,\sep/4)
		(m)--++(0,-\sep/8)
		node[circle,draw,below] (x){$x$};
	\end{sd}
=
	\begin{sd}
	\node[circle,draw] (A) {$A$};
	\draw (A)to[out=90,in=90] coordinate[midway] (m) 
		([xshift=\sep*0.6]A.north) node[circle,draw,below] (x){$x$};
	\draw (m)--++(0,\sep/4) (A.south)--++(0,-\sep/8);
	\end{sd}
-
	\begin{sd}
	\coordinate (B);
	\draw (B)--++(0,-\sep/4);
	\draw (B)to[out=90,in=90] coordinate[midway] (m) 
		([xshift=\sep/2]B.north) node[circle,draw,below] (x){$x$};
	\draw (m)--++(0,\sep/8)node[circle,draw,above] (A) {$A$}
	(A.north)--++(0,\sep/4);
	\end{sd}
=[\rho(x),A].
\]
\end{proof}

Recall the one-to-one correspondence \eqref{graph-rel} between real quantum graphs $A$ on $(B,\psi)$ and `edge space' $B$-$B$-bimodules $\mathcal{S}=\ran P_A \subset B \otimes B$ represented by orthogonal projection $P_A$ onto $\mathcal{S} \subset L^2(B,\psi)^{\otimes 2}$.
Similarly to the classical case, $\nabla_A$ is a map to the edge space.

\begin{prop}
Let $(B,\psi,A)$ be a real quantum graph. Then the following holds.
\begin{description}
\item[$(1)$]
The range of $\nabla_A$ is included in $\ran P_A$, i.e., 
$P_A \nabla_A=\nabla_A$.
\item[$(2)$]
The operator $\nabla_A$ is a $\C$-derivation, i.e., $\nabla_A(xy)=(\nabla_A x)y+x(\nabla_A y)$ for all $x,y \in B$ and $\nabla_A(\lambda)=0$ for any $\lambda \in \C \subset B$.
\end{description}
\end{prop}

\begin{proof}
$(1)$
Note that the real condition \eqref{realness} implies
\begin{align}
\begin{sd}
\node[circle,draw](A){$A^\dagger$};
\draw (A)to[out=-90,in=-90] coordinate[midway](m) 
		([xshift=\sep/2]A.south)
		--([xshift=\sep/2]A.north)--++(0,\sep/4)
		(A.north)--++(0,\sep/4)
		(m)--++(0,-\sep/4);
\end{sd}
&=
	\begin{sd}
	\node[circle,draw] (G) {$A$};
	\draw (G)to[out=-90,in=-90] coordinate[midway] (m*) 
		([xshift=-\sep/2]G.south);
	\draw (G)to[out=90,in=90] coordinate[midway] (m) 
		([xshift=\sep/2]G.north);
	\draw (m)--++(0,\sep/4) (m*)
		([xshift=-\sep/2]G.south)--++(0,\sep/2) 
		([xshift=\sep/2]G.north)--++(0,-\sep/2);
	\end{sd}
=\delta^2 P_A(1 \otimes \cdot);
&
\begin{sd}
\node[circle,draw](A){$A$};
\draw (A)to[out=-90,in=-90] coordinate[midway](m) 
		([xshift=-\sep/2]A.south)
		--([xshift=-\sep/2]A.north)--++(0,\sep/4)
		(A.north)--++(0,\sep/4)
		(m)--++(0,-\sep/4);
\end{sd}
&=
	\begin{sd}
	\node[circle,draw] (G) {$A$};
	\draw (G)to[out=-90,in=-90] coordinate[midway] (m*) 
		([xshift=-\sep/2]G.south);
	\draw (G)to[out=90,in=90] coordinate[midway] (m) 
		([xshift=\sep/2]G.north);
	\draw (m)--++(0,\sep/4) (m*)--++(0,-\sep/4) 
		([xshift=-\sep/2]G.south)--++(0,\sep/2) 
		([xshift=\sep/2]G.north) arc(90:-270:\unode);
	\end{sd}
=\delta^2 P_A(\cdot \otimes 1).
\end{align}
Thus we have $\nabla_A=P_A(1 \otimes \cdot-\cdot \otimes 1)$, and 
\[
P_A \nabla_A=P_A^2(1 \otimes \cdot-\cdot \otimes 1)=\nabla_A
\] 
by idempotence of $P_A$.

\noindent$(2)$ 
Now we have $\nabla_A 1= P_A(1 \otimes 1-1 \otimes 1)=0$.
It remains to show $\nabla_A(xy)=(\nabla_A x)y+x(\nabla_A y)$ for $x,y \in B$. 
Indeed by bimodule property $P_A(xzy)=x(P_A z)y$ for $x,y \in B, z \in B^{\otimes 2}$, we obtain
\begin{align}
\nabla_A(xy)
&= P_A(1 \otimes xy - xy \otimes 1)
\\
&= P_A(1 \otimes xy - x \otimes y + x \otimes y - x y \otimes 1)
\\
&= P_A((1 \otimes x - x \otimes 1) y + x (1 \otimes y - y \otimes 1))
\\
&=(\nabla_A x)y+x(\nabla_A y).
\end{align}
\end{proof}

\begin{dfn}[{Generalization of Ganesan \cite{Ganesan2023spectral} to directed graphs}]
Let $\mathcal{G}=(B,\psi,A)$ be a quantum graph.
We define the (left) indegree matrix $D_\mathrm{in}:B \to B$ and (right) outdegree matrix $D_\mathrm{out}:B \to B$ by 
\begin{align}
D_\mathrm{in}&=\lambda(A1_B)
=
\begin{sd}
	\path (\sep/3,0)node[circle,draw](A2){$A$};
	\draw (A2.north)to[out=90,in=90]coordinate[midway](m1)
		([xshift=0.5cm]A2.north) --++(0,-\sep/2) 
		(m1)--++(0,\sep/4)
		(A2)--++(0,-\sep/2)arc(-270:90:\unode);
\end{sd};
& D_\mathrm{out}&=\rho(A^\dagger 1_B)
=
\begin{sd}
	\path (\sep/3,0)node[circle,draw](A2){$A^\dagger$};
	\draw (A2.north)to[out=90,in=90]coordinate[midway](m1)
		([xshift=-0.5cm]A2.north) --++(0,-\sep/2) 
		(m1)--++(0,\sep/4)
		(A2)--++(0,-\sep/2)arc(-270:90:\unode);
\end{sd}
\overset{(\textrm{if }A\textrm{:real})}{=}
\begin{sd}
	\path (-\sep/3,0)node[circle,draw](A1){$A$};
	\draw (A1.south)to[out=-90,in=-90]coordinate[midway](m1)
		([xshift=-0.5cm]A1.south) --++(0,\sep/2) 
		(m1)--++(0,-\sep/4) 
		(A1)--++(0,\sep/2)arc(-90:270:\unode);
\end{sd}
\end{align}
where $\lambda$ (resp. $\rho$) is the left (resp. right) multiplication.
\end{dfn}

If $\mathcal{G}$ is undirected, then $D_\mathrm{out}=D_\mathrm{in}^*$ by
\[
D_\mathrm{in}^* x	 = ((A1) x^*)^*
= x (A1)^* = x (A^* 1) 
\overset{\textrm{(undirected)}}{=} x (A^\dagger 1)
=D_\mathrm{out} x.
\]
And $D_\mathrm{out}=D_\mathrm{in}=d \,\id_B$ if $\mathcal{G}$ is $d$-regular.

\begin{lem}\label{lem:specbound}
Let $\mathcal{G}=(B,\psi,A)$ be a real quantum graph. Then
\[
0 \leq \nabla_A^\dagger \nabla_A=\delta^{-2} \left(D_\mathrm{in} - A + D_\mathrm{out} - A^\dagger \right).
\]
Moreover if $\mathcal{G}$ is $d$-regular, 
\[
0 \leq \nabla_A^\dagger \nabla_A=2\delta^{-2} \left(d \, \id_B- \frac{A+A^\dagger}{2} \right).
\]
In particular 
$\frac{\theta A+\ol{\theta}A^\dagger}{2} \leq d \, \id_B$ for all $\theta \in \T$.
\end{lem}

\begin{proof}
We can compute directly
\begin{align}
\nabla^\dagger \nabla 
&=\delta^{-4} m(A \otimes \id - \id \otimes A^\dagger)(A^\dagger \otimes \id - \id \otimes A)m^\dagger
\\
&=
\delta^{-4}
\left(
\begin{sd}
	\coordinate (H1) at (0,-3/5*\sep) ;
	\coordinate (H3) at (0,3/5*\sep);
	\path (-\sep/3,0)node[draw](A1){$AA^\dagger$} 
		(\sep/3,0)node(A2){$\phantom{A}$};
	\draw (A2.south) to[out=-90,in=-90] coordinate[midway] (m2) (A1.south) ;
	\draw (A2.north) to[out=90,in=90] coordinate[midway] (m3) (A1.north) ;
	\draw (m2)--(H1) (m3)--(H3) (A2.north)--(A2.south);
\end{sd}
+
\begin{sd}
	\coordinate (H1) at (0,-3/5*\sep) ;
	\coordinate (H3) at (0,3/5*\sep);
	\path (-\sep/3,0)node(A1){$\phantom{A}$} 
		(\sep/3,0)node[draw](A2){$A^\dagger A$};
	\draw (A2.south) to[out=-90,in=-90] coordinate[midway] (m2) (A1.south) ;
	\draw (A2.north) to[out=90,in=90] coordinate[midway] (m3) (A1.north) ;
	\draw (m2)--(H1) (m3)--(H3) (A1.north)--(A1.south);
\end{sd}
-
\begin{sd}
	\coordinate (H1) at (0,-3/5*\sep) ;
	\coordinate (H3) at (0,3/5*\sep);
	\path (-\sep/3,0)node[circle,draw](A1){$A$} 
		(\sep/3,0)node[circle,draw](A2){$A$};
	\draw (A2.south) to[out=-90,in=-90] coordinate[midway] (m2) (A1.south) ;
	\draw (A2.north) to[out=90,in=90] coordinate[midway] (m3) (A1.north) ;
	\draw (m2)--(H1) (m3)--(H3);
\end{sd}
-
\begin{sd}
	\coordinate (H1) at (0,-3/5*\sep) ;
	\coordinate (H3) at (0,3/5*\sep);
	\path (-\sep/3,0)node[circle,draw](A1){$A^\dagger$} 
		(\sep/3,0)node[circle,draw](A2){$A^\dagger$};
	\draw (A2.south) to[out=-90,in=-90] coordinate[midway] (m2) (A1.south) ;
	\draw (A2.north) to[out=90,in=90] coordinate[midway] (m3) (A1.north) ;
	\draw (m2)--(H1) (m3)--(H3);
\end{sd}
\right)
\\
&=
\delta^{-2}
\left(
\delta^{-2}
\begin{sd}
	\coordinate (H1) at (0,-3/5*\sep) ;
	\coordinate (H3) at (0,3/5*\sep);
	\path (-\sep/3,0)node[circle,draw](A1){$A$} 
		(\sep/3,0)node[circle,draw](A2){$A$};
	\draw (A2.south) to[out=-90,in=-90] coordinate[midway] (m2) (A1.south) ;
	\draw (A2.north) to[out=90,in=90] coordinate[midway] (m3) (A1.north) ;
	\draw (m3)to[out=90,in=90]coordinate[midway](m1)
		([xshift=0.3cm]A2.north) --++(0,-\sep/2) 
		(m1)--++(0,\sep/4);
\end{sd}
+
\delta^{-2}
\begin{sd}
	\coordinate (H1) at (0,-3/5*\sep) ;
	\coordinate (H3) at (0,3/5*\sep);
	\path (-\sep/3,0)node[circle,draw](A1){$A$} 
		(\sep/3,0)node[circle,draw](A2){$A$};
	\draw (A2.south) to[out=-90,in=-90] coordinate[midway] (m2) (A1.south) ;
	\draw (A2.north) to[out=90,in=90] coordinate[midway] (m3) (A1.north) ;
	\draw (m2)to[out=-90,in=-90]coordinate[midway](m1)
		([xshift=-0.3cm]A1.south) --++(0,\sep/2) 
		(m1)--++(0,-\sep/4);
\end{sd}
-A-A^\dagger
\right)
\\
&=
\delta^{-2}
\left(
\begin{sd}
	\path (\sep/3,0)node[circle,draw](A2){$A$};
	\draw (A2.north)to[out=90,in=90]coordinate[midway](m1)
		([xshift=0.5cm]A2.north) --++(0,-\sep/2) 
		(m1)--++(0,\sep/4)
		(A2)--++(0,-\sep/2)arc(-270:90:\unode);
\end{sd}
+
\begin{sd}
	\path (-\sep/3,0)node[circle,draw](A1){$A$};
	\draw (A1.south)to[out=-90,in=-90]coordinate[midway](m1)
		([xshift=-0.5cm]A1.south) --++(0,\sep/2) 
		(m1)--++(0,-\sep/4) 
		(A1)--++(0,\sep/2)arc(-90:270:\unode);
\end{sd}
-A-A^\dagger
\right)
\\
&=\delta^{-2} \left(D_\mathrm{in} - A + D_\mathrm{out} - A^\dagger \right).
\end{align}
If it is $d$-regular, then $D_\mathrm{out}=D_\mathrm{in}=d \,\id_B$ yields
\[
\nabla^\dagger \nabla 
=2\delta^{-2} \left(d \, \id_B- \frac{A+A^\dagger}{2} \right).
\]
Replacing $A$ by $\lambda A$ and $A^\dagger$ by $\ol{\lambda} A^\dagger$ in $\nabla_A$, we deduce  
$d \, \id_B - \frac{\lambda^2 A+\ol{\lambda}^2 A^\dagger}{2} \geq 0$ for any $\lambda \in \T$. Since $\theta=\lambda^2$ ranges all $\theta \in \T$, we obtain $\frac{\theta A+\ol{\theta}A^\dagger}{2} \leq d \, \id_B$.
\end{proof}

\begin{rmk}
Ganesan \cite{Ganesan2023spectral} defined the graph Laplacian $L$ by $L=D-A$ for undirected quantum graphs with right degree matrix $D=D_\mathrm{out}$.
In this case, our Laplacian $\Delta\coloneqq\delta^{2}\nabla^\dagger \nabla=L^* +L$ is a `double' of usual Laplacian.
Usually, the gradient of an undirected classical graph is defined by the gradient as in Lemma \ref{lem:classicalgrad} of an orientation (i.e., a half) of the original graph. This is why we obtained the doubled Laplacian, and such duplication is inevitable because quantum graphs do not always have an orientation as shown below. 
\end{rmk}

\begin{dfn}
Let $\mathcal{G}=(B,\psi,A)$ be an undirected quantum graph. We say that a Schur projection $T: B\to B$ is an orientation of $\mathcal{G}$ if $A\bullet T=T$, $T\bullet T^\dagger=0$, and $\ran (T\bullet \cdot)+\ran (T^\dagger \bullet \cdot)=\ran (A\bullet \cdot)$. 
\end{dfn}
Note that this definition is equivalent to $A=T+T^\dagger$ if $T^\dagger$ is also a Schur projection.
The definition states that $T$ is a directed subgraph (edge subset) of $A$ over the same quantum set, $T^\dagger$ is the opposite orientation of $T$, and $T$ and $T^\dagger$ disjointly cover $A$.

If $\mathcal{G}$ is non-tracial, then the GNS adjoint $T^\dagger$ is not always real, hence not necessarily a Schur projection and $A=T+T^\dagger$ may not hold. 
To avoid such a problem, we can instead consider 
a KMS symmetric quantum graph $\mathcal{G}$ and the KMS adjoint $T^\ddagger$ to define an orientation simply by $A=T+T^\ddagger$.

\begin{ex}[A non-orientable quantum graph]
Consider $1$-regular irreflexive undirected quantum graph $\mathcal{G}=(M_2, \tau=\Tr/2, A=2E_{\C^2}-\id_{M_2}:\begin{pmatrix}a & b \\ c & d \end{pmatrix}
\mapsto \begin{pmatrix}a & -b \\ -c & d \end{pmatrix})$
(c.f.~\cite{Gromada2022some,Matsuda2022classification}),
 where $E_{\C^2}$ is the conditional expectation onto the diagonal subalgebra.
Its corresponding projection 
\[
p_A=\frac{1}{2}\begin{pmatrix}
1&0&0&-1\\
0&0&0&0\\
0&0&0&0\\
-1&0&0&1
\end{pmatrix}
\in M_2^\op \otimes M_2 =M_4
\]
is rank one, hence it does not have an orientation $T$ whose corresponding projection $p_T$ must satisfy $p_A=p_T+p_{T^\dagger}$.
\end{ex}

\subsection{Spectral bound by the degree}

\begin{prop}\label{prop:specrad=d}
Let $\mathcal{G}$ be a $d$-regular real quantum graph. The spectral radius $r(A)$ of the adjacency operator satisfies $r(A)=d$.
\end{prop}

\begin{proof}
For a nonzero $\lambda \in \Spec(A)$ and a unit eigenvector $x \in \ker(\lambda \,\id_B -A)$, 
choose $\theta \in \T$ so that $\theta\lambda=\abs{\lambda}$.
Then Lemma \ref{lem:specbound} shows
\[
d = d \braket{x|x} \geq  \frac{\braket{x| \theta A+\ol{\theta}A^\dagger |x}}{2}
= \frac{\theta \braket{x|Ax}+\ol{\theta}\braket{Ax|x}}{2}
= \frac{\theta \lambda + \ol{\theta \lambda}}{2}
=\abs{\lambda}.
\]
Thus $r(A) = \sup_{\lambda \in \Spec(A)} \abs{\lambda} \leq d$.
Since $d \in \Spec(A)$, we have $r(A)=d$.
\end{proof}

\begin{thm}\label{thm:norm=deg}
Let $\mathcal{G}=(B,\psi,A)$ be a $d$-regular quantum graph. Then the identity of the operator norm on $B(L^2(\mathcal{G}))$ and the degree
\[
\norm{A}_\op=d
\]
 holds if either of the following is satisfied:
\begin{description}
\item[$(1)$]
$\mathcal{G}$ is undirected, whence $\Spec(A)\subset [-d,d]$;
\item[$(2)$]
both $A$ and $A^\dagger$ are real;
\item[$(3)$]
$\mathcal{G}$ is real and tracial, i.e., $A$ is real and $\psi=\tau_B$.
\end{description}
\end{thm}

\begin{proof}
$(1)$ Since $A$ is normal $AA^\dagger=A^\dagger A$, Proposition \ref{prop:specrad=d} implies 
$\norm{A}_\op=r(A)=d$. 
Thus self-adjointness shows $\Spec(A)\subset [-d,d]$.

\noindent$(2)$ We prove this by embedding $A$ into an undirected $d$-regular quantum graph 
\[
\left( B\otimes \C^2, \tilde{\psi}=\psi \otimes \tau_{\C^2}, 
\tilde{A} \coloneqq A\otimes E_{12}+A^\dagger \otimes E_{21} \right)
=
\left( B\oplus B, \frac{\psi \oplus \psi}{2}, \begin{pmatrix}
0 & A \\ A^\dagger & 0
\end{pmatrix} \right)
\]
 where $E_{ij}$ are matrix units in $M_2$.
By definition $\tilde{A}$ is self-adjoint. 
Note that $A,A^\dagger$ are quantum graphs on $(B,\psi)$ and $E_{ij}$ are (quantum) graphs on $(\C^2,\tau_{\C^2})$.
Then $A \otimes E_{12}$ and $A^\dagger \otimes E_{21}$ are quantum graphs on $(B\otimes\C^2,\psi \otimes \tau_{\C^2})$. Since the Schur product of $E_{12}$ and $E_{21}$ is zero, 
$\tilde{A} = A \otimes E_{12} + A^\dagger \otimes E_{21}$
 is also a quantum graph.

By assumption, $A$ and $A^\dagger$ are real. So are $A \otimes E_{12}$ and $A^\dagger \otimes E_{21}$, hence $\tilde{A}$ is real.

The regularity follows from 
$\tilde{A}(1_B \otimes 1_{\C^2})=d1_B\otimes e_1 +d1_B\otimes e_2 
=d (1_B \otimes 1_{\C^2})$.

Therefore we have
\[
d=\norm{\tilde{A}}_{B(L^2(B\otimes\C^2))}
\geq \norm{\tilde{A}|_{L^2(B)\otimes e_2 \to L^2(B)\otimes e_1}}
= \norm{A}_{B(L^2(B))}
\]
via isometric identifications 
$L^2(B) \ni x \mapsto x \otimes \sqrt{2}e_i \in L^2(B)\otimes e_i$ for $i=1,2$. 
By $d \in\Spec(A)$, we obtain $\norm{A}=d$.

\noindent$(3)$ By traciality, $A^\dagger$ is also real. Thus $(3)$ follows from $(2)$.
\end{proof}

\begin{cor}\label{thm:KMSnorm=deg}
Let $\mathcal{G}=(B,\psi,A)$ be a $d$-regular quantum graph. 
Then we have the identity of the degree and the operator norm with respect to the KMS inner product on $B$:
\[
\norm{A}_\op=d.
\]
\end{cor}

\begin{proof}
By \eqref{Treal=Tddreal}, the realness of $A$ implies that $A^\ddagger$ is also real. Thus we have the KMS version of Theorem \ref{thm:norm=deg} (2): both $A$ and $A^\ddagger$ are real.
Since the spectral radius does not depend on the inner product structure, we have $r(A)=d$ by Proposition \ref{prop:specrad=d}.
Therefore by the same argument as in the proof of Theorem \ref{thm:norm=deg}, 
we obtain $\norm{A}_\op=d$ over the KMS Hilbert space $B$.
\end{proof}

\begin{cor}
Let $\mathcal{G}=(B,\psi,A)$ be a $d$-regular undirected irreflexive quantum graph. Then $\Spec(A)\subset [-d,d]$ and $0 \leq d \leq \delta^2-1$.
Equivalently if $\mathcal{G}$ is a $d$-regular undirected reflexive quantum graph, then $\Spec(A)\subset [-d+2,d]$ and $1 \leq d \leq \delta^2$.
\end{cor}

\begin{proof}
If $\mathcal{G}$ is irreflexive, then $\Spec A \subset [-d,d]$ follows from Theorem \ref{thm:norm=deg}. Its reflexive version is given by $(B,\psi,A+\id)$ as a $(d+1)$-regular undirected quantum graph, hence Lemma \ref{lem:deg<dim} shows that $0 \leq d \leq \delta^2-1$. 
If $\mathcal{G}$ is reflexive, we may replace $d$ in the previous argumant by $d-1$ and obtain $\Spec (A-\id) \subset [-d+1,d-1]$, i.e., $\Spec A \subset [-d+2,d]$, and $1 \leq d \leq \delta^2$.
\end{proof}

\begin{open*}
In view of the above, we wonder if $\norm{A}_\op=d$ holds with a weaker assumption with respect to the GNS inner product.

Although we showed that some irreflexive quantum graphs do not admit an orientation, there may be a better definition that makes any irreflexive undirected quantum graphs orientable.

As we have the quantum graph Laplacians $\Delta=\delta^2 \nabla^\dagger \nabla$ and $L$, it is natural to consider a quantum Markov semigroup $e^{-t \Delta}$, which is the heat semigroup over the quantum graph. 
We leave it as an open question for future work to investigate the property of $e^{-t \Delta}$ such as the complete logarithmic Sobolev inequality (cf.~\cite{Brannan2022complete}).
\end{open*}

\section{Characterization of graph properties}

In this section, we introduce graph homomorphisms respecting the adjacency matrices and define the connectedness and bipartiteness of quantum graphs in terms of graph homomorphisms. After that, we give algebraic characterizations of these properties for regular quantum graphs. 

\subsection{Graph properties defined by homomorphisms}

\begin{dfn}\label{dfn:adjhom}
Let $\mathcal{G} = (B,\psi,A), \mathcal{G}' = (B',\psi',A')$ be quantum graphs.
A graph homomorphism $f^\op : \mathcal{G} \to \mathcal{G}'$ is a unital $*$-homomorphism $f: B' \to B$ satisfying $A' \bullet (f^\dagger A f) =f^\dagger A f$.

We say that $f^\op$ is surjective if $f$ is injective, and $f^\op$ is injective if $f$ is surjective.
\end{dfn}

This definition states that the pushforward $f^\dagger A f$ of the adjacency matrix of $\mathcal{G}$ is in the edges of $\mathcal{G}'$.

\begin{dfn}
Let $\mathcal{G}$ be a quantum graph.
\begin{itemize}
\item
 $\mathcal{G}$ is \emph{disconnected} if there is a surjective graph homomorphism $\mathcal{G} \to T_2$;
\item
 $\mathcal{G}$ is \emph{connected} if it is not disconnected, i.e., there is no surjective graph homomorphism $\mathcal{G} \to T_2$;
\item
 $\mathcal{G}$ is \emph{bipartite} if there is a surjective graph homomorphism $\mathcal{G} \to K_2$;
\item
 $\mathcal{G}$ \emph{has a bipartite component} if there is a  graph homomorphism $\mathcal{G} \to K_2 \sqcup T_1$ that is onto $K_2$, i.e., 
there is a unital $*$-homomorphism $f: \C^2 \oplus \C \to B$ satisfying 
$\begin{pmatrix}
0 & 1 & 0 \\ 1 & 0 & 0 \\ 0 & 0 & 1
\end{pmatrix}  \bullet f^\dagger A f = f^\dagger A f$ and $f$ is injective on $\C^2 \oplus 0$.
\end{itemize}
\end{dfn}

If $\mathcal{G}=(V,E)$ is classical, these definitions agree with classical definitions:
$\mathcal{G}$ is disconnected (resp. bipartite) if there is a decomposition $V=V_0 \sqcup V_1$ with no edges between $V_0$ and $V_1$ (resp. with all edges between $V_0$ and $V_1$). The equivalence is proved by mapping $V_0$ and $V_1$ to the distinct vertices of $K_2$ or $T_2$.

A naive definition of these properties by $A=\begin{pmatrix} * & 0  \\ 0 & *  \end{pmatrix}$ or $\begin{pmatrix} 0 & *  \\ * & 0  \end{pmatrix}$
along some nontrivial decomposition $B=B_0\oplus B_1$ 
of the quantum set
is too restrictive for quantum graphs. 
Indeed there exists a $1$-regular undirected irreflexive quantum graph $(M_2,\tau,A=2E_{\C^2}-\id)$
(c.f.~\cite{Gromada2022some,Matsuda2022classification})
 with $\Spec A=\{-1,-1,1,1\}$, which looks like bipartite and disconnected but has no nontrivial decomposition $M_2=B_0\oplus B_1$.
That is why we defined as above.

The following is the key lemma to prove spectral characterizations of these properties.
This lemma allows us to control the decomposition of a self-adjoint operator into positive and negative parts.

\begin{lem}\label{lem:posnegdec}
Let $B$ be a $C^*$-algebra with a faithful state $\psi$, and $x_\pm, y_\pm \in B$ be positive elements satisfying 
\[
x_+ - x_- = y_+ - y_-
\]
with $\psi(x_+) = \psi(y_+), \psi(x_-) = \psi(y_-)$.
Assume that there is a projection $p \in B$ such that 
\[
p x_+=x_+=x_+ p, \quad (1-p) x_-=x_-=x_- (1-p), 
\quad \psi(p \ \cdot)=\psi(\cdot \ p).
\]  
Then it follows that
$x_+=y_+, x_-=y_-$.
\end{lem}

\begin{proof}
We show that $\xi \coloneqq  y_+ - x_+=y_- - x_-$ is zero.
By assumptions on $p$, we have
\begin{align}
p \xi p &= p y_+ p - x_+ = p y_- p \geq 0
\\
(1-p) \xi (1-p)&= (1-p) y_- (1-p) -  x_- = (1-p) y_+ (1-p) \geq 0
\\
p \xi (1-p) &= p y_+ (1-p) = p y_- (1-p).
\end{align}
By $\psi(\xi)=\psi(y_+)-\psi(x_+)=0$ and $\psi(p\xi(1-p))=\psi(\xi(1-p)p)=0$, we have 
\begin{align}
0 = \psi(\xi) &= \psi(p \xi p) + \psi((1-p) \xi (1-p)) +\psi(p \xi (1-p)) +\psi((1-p) \xi p)
\\
&= \psi(p \xi p) + \psi((1-p) \xi (1-p)).
\end{align} 
Since $p \xi p$ and $(1-p) \xi (1-p)$ are positive, faithfulness of $\psi$ implies
\begin{align}
p\xi p = (1-p) \xi (1-p) =0.
\end{align}
By positivity of $y_+=x_+ + \xi$, it follows for all $t \in \R$ that 
\[
(p + t(1-p))y_+ (p + t(1-p))=x_+ + t(p \xi (1-p) + (1-p) \xi p)
\]
 is positive. 
Since $p \xi (1-p) + (1-p) \xi p$ is self-adjoint, if it has a nonzero positive or negative part, $x_+ + t(p \xi (1-p) + (1-p) \xi p)$ cannot be always positive.
Therefore $p \xi (1-p) + (1-p) \xi p=0$,
hence $\xi=(p + (1-p))\xi (p + (1-p))=0$.
\end{proof}

\begin{lem}\label{lem:tracialposneg}
Let $B$ be a von Neumann algebra with a faithful tracial state $\tau$, and $x_\pm, y_\pm \in B$ be positive elements satisfying 
\[
x_+ - x_- = y_+ - y_-, \quad x_+ x_-=x_- x_+=0
\]
with $\tau(x_+) = \tau(y_+), \tau(x_-) = \tau(y_-)$. 
Then it follows that
$x_+=y_+, x_-=y_-$.
\end{lem}

\begin{proof}
Since $x_+ x_-=x_- x_+=0$, the range projection $p$ of $x_+$ satisfies
$p x_+=x_+=x_+ p$ and $(1-p) x_-=x_-=x_- (1-p)$. 
Since $\tau$ is tracial, we also have $\tau(p \ \cdot)=\tau(\cdot \ p)$. 
Thus Lemma \ref{lem:posnegdec} shows $x_+=y_+, x_-=y_-$.
\end{proof}

\begin{rmk}
Note that the assumption $\psi(p \ \cdot)=\psi(\cdot \ p)$ 
is essential in Lemma \ref{lem:posnegdec}. 
Indeed we have the following counterexample without this property. 
Let $B=M_2, \psi=\omega_q \circ \ad(u)=\Tr(u^* Qu \cdot)$ where 
$\displaystyle 
Q=\frac{1}{1+q^2} \begin{pmatrix} 1 & 0 \\ 0 & q^2 \end{pmatrix}, q \in (0,1), 
u=\frac{1}{\sqrt{2}} \begin{pmatrix} 1 & -1 \\ 1 & 1 \end{pmatrix}$.
Put 
\begin{align}
x_+ &= \begin{pmatrix} 1 & 0 \\ 0 & 0 \end{pmatrix} \geq 0,
&
x_- &= \begin{pmatrix} 0 & 0 \\ 0 & 1 \end{pmatrix} \geq 0,
&
\xi &= \alpha \begin{pmatrix} 
	1 & \frac{1+q^2}{1-q^2} \\
	 \frac{1+q^2}{1-q^2} & 1 
	\end{pmatrix} : \text{s.a.},
\end{align}
for $\alpha \in \left(0, \frac{(q^{-1}-q)^2}{4} \right]$.
It follows that 
\begin{align}
y_\pm &= x_\pm +\xi \geq 0, 
&
\psi(\xi) &= 0
\text{, i.e., }
\psi(x_\pm) = \psi(y_\pm),
\end{align}
and $x_+, x_-$ are orthogonal projections, but $\xi \neq 0$.
\end{rmk}

\begin{proof}
We have
\[
y_+=\begin{pmatrix} 
	1+\alpha & \frac{1+q^2}{1-q^2}\alpha \\
	 \frac{1+q^2}{1-q^2}\alpha & \alpha 
	\end{pmatrix},
\]
hence 
$\Tr(y_+)=1+2\alpha>0$
and
\[
\det y_+ = \alpha +\alpha^2 \left(1 - \left(\frac{1+q^2}{1-q^2}\right)^2 \right)
= \alpha \left(1 - \alpha \frac{4 q^2}{(1-q^2)^2} \right) \geq 0
\]
show that $y_+ \geq 0$, and $y_- \geq 0$ as well.
By simple computation, we get
\begin{align}
\psi(\xi)=\Tr(u^* Qu \xi)
=\frac{\alpha}{2} \Tr \left( 
	\begin{pmatrix} 
	1 & \frac{-1+q^2}{1+q^2} \\
	 \frac{-1+q^2}{1+q^2} & 1
	\end{pmatrix}
	\begin{pmatrix} 
	1 & \frac{1+q^2}{1-q^2} \\
	 \frac{1+q^2}{1-q^2} & 1 
	\end{pmatrix} 
\right)
=0.
\end{align}
\end{proof}

\begin{lem}\label{lem:fcncalcev}
Let $\mathcal{G} = (B,\psi,A)$ be a $d$-regular undirected tracial quantum graph.
It follows for any self-adjoint $x \in \ker (d \,\id -A)$ that $C^*(x) \subset \ker (d \,\id -A)$.
\end{lem}

\begin{proof}
It suffices to show that $A p_i = d p_i$ 
for the spectral projections $\{p_1,..., p_k \}$ 
of $x= \sum_{i=1}^k \lambda_i p_i$
with $\lambda_1 > \cdots > \lambda_k$.
Consider $\ker (d\, \id -A) \ni x-\lambda_2 1_B 
= (\lambda_1 - \lambda_2) p_1 - \sum_{i=2}^k (\lambda_2 - \lambda_i) p_i$,
then 
\[
(\lambda_1 - \lambda_2) Ap_1 - \sum_{i=2}^k (\lambda_2 - \lambda_i) Ap_i
= d(\lambda_1 - \lambda_2) p_1 - d \sum_{i=2}^k (\lambda_2 - \lambda_i) p_i.
\]
Since $(\lambda_1 - \lambda_2) p_1$ and $\sum_{i=2}^k (\lambda_2 - \lambda_i) p_i$ are positive and have disjoint supports, $\psi A=d \psi$ shows that we can apply Lemma \ref{lem:tracialposneg}. Thus
\[
(\lambda_1 - \lambda_2) Ap_1 = d(\lambda_1 - \lambda_2) p_1,
\]
hence $p_1, \sum_{i=2}^k \lambda_i p_i \in \ker(d \,\id -A)$. 
Inductively we get $p_1,..., p_k \in \ker(d \,\id -A)$.

\end{proof}

\subsection{Connected quantum graphs}

\begin{thm}\label{thm:conn=algconn}
Let $\mathcal{G} = (B,\psi,A)$ be a $d$-regular undirected tracial quantum graph. The following are equivalent:
\begin{enumerate}
\item [$(1)$]
$\mathcal{G}$ is connected.

\item [$(2)$]
$d \in \Spec(A)$ is a simple root, i.e., $\dim \ker(d \,\id -A)=1$.
\end{enumerate}
\end{thm}

\begin{proof} 
If $\dim B=1$, then $\mathcal{G}$ is connected and $d$ is simple.
If ${\dim B}\geq 2$ and $d=0$, then $A=0$ by Lemma \ref{lem:numofedges} and $d$ has multiplicity $\geq 2$, whence there is an injective unital $*$-homomorphism $f: \C^2 \to B$. Hence neither $\mathcal{G}$ is connected nor $d$ is simple.
In the sequel of the proof, we may assume $d>0$ and $\dim B \geq 2$.

\noindent $((2) \implies (1))$: 
We show that $d$ is a multiple root if $\mathcal{G}$ is disconnected.

We have an injective unital $*$-homomorphism $f: \C^2 \to B$ such that 
$\begin{pmatrix} 1 & 0  \\ 0 & 1  \end{pmatrix} \bullet  f^\dagger A f
=f^\dagger A f$. Put $x_1=f(e_1), x_2=f(e_2) \in B$, 
which are mutually orthogonal nonzero projections satisfying $x_1+ x_2 = 1_B$.
The regularity shows $Ax_1 + Ax_2=A1_B=dx_1 + dx_2$.
By $(f^\dagger A f)_{ij}=2\braket{e_i|f^\dagger A f |e_j}=2\braket{x_i| A x_j}$ and  
$\begin{pmatrix} 1 & 0  \\ 0 & 1  \end{pmatrix} \bullet f^\dagger A f 
=f^\dagger A f$, it follows that 
\begin{align}
\braket{x_1| A x_2} &= \braket{x_2| A x_1} = 0;
\\
\braket{x_1| A x_1} &= \braket{x_1+x_2| A x_1} = \psi(A x_1) = d\psi(x_1);
\\
\braket{x_2| A x_2} &= \braket{x_1+x_2| A x_2} = \psi(A x_2) = d\psi(x_2).
\end{align}
Thus $Ax_1= dx_1 + (dx_2-Ax_2)$ gives the orthogonal decomposition of $Ax_1$
 along $\C x_1 \oplus (x_1)^\perp$. Then we have
\[
d^2 \psi(x_1) \geq \norm{Ax_1}_2^2
= \norm{d x_1}_2^2 + \norm{dx_2-Ax_2}_2^2
= d^2 \psi(x_1) + \norm{dx_1-Ax_2}_2^2,
\]
hence $\norm{dx_2-Ax_2}_2=0$, i.e., $Ax_2=dx_2$ and $Ax_1=dx_1$.
Therefore $d \in \Spec(A)$ has multiplicity more than $1$. 

\noindent $((1) \implies (2))$:
We show that $\mathcal{G}$ is disconnected if $d$ is not simple.

By Lemma \ref{lem:saevreal} and the multiplicity of $d$, there is a self-adjoint $x \in \ker(d\, \id -A) \setminus \C1$, 
and Lemma \ref{lem:fcncalcev} allows us to take mutually orthogonal projections $x_1,x_2 \in \ker(d\, \id -A)$
satisfying $x_1+ x_2=1$ as spectral projections of $x$. 
Thus we obtain an injective $*$-homomorphism
$f: \C^2 \to B$ defined by $f(e_i)=x_i$ for $i=1,2$.
It satisfies
\begin{align}
2 \braket{e_i|f^\dagger Af | e_j} 
&= 2 \braket{x_i | A x_j} = 2d \braket{x_i | x_j} =2d \psi(x_i) \delta_{ij}.
\end{align}
Thus $f^\dagger Af = \begin{pmatrix} 2d \psi(x_1) & 0  \\ 0 & 2d \psi(x_2)  \end{pmatrix}$,
which gives a surjective graph homomorphism $f^\op: \mathcal{G} \to T_2$.
\end{proof}

\subsection{Bipartite quantum graphs}

\begin{thm}\label{thm:connbipartite}
Let $\mathcal{G} = (B,\psi,A)$ be a $d$-regular connected undirected tracial quantum graph. The following are equivalent:
\begin{enumerate}
\item [$(1)$]
$\mathcal{G}$ is bipartite.

\item [$(2)$]
$-d \in \Spec(A)$.
If $d=0$, we require that the multiplicity of $0 \in \Spec(A)$ is at least two, i.e., ${\dim B} \geq 2$.
\end{enumerate}
\end{thm}

\begin{proof} 
If  ${\dim B} =1$, then $A=0$ with simple root $d=0$ or $A=\id_\C$ with $d=1\neq -d$, hence neither bipartite nor $-d \in \Spec(A)$.
We may assume ${\dim B}\geq 2$, then the connectedness implies $d>0$ as argued in the proof of Theorem \ref{thm:conn=algconn}.

\noindent $((1) \implies (2))$:
We have an injective unital $*$-homomorphism $f: \C^2 \to B$ such that 
$\begin{pmatrix} 0 & 1  \\ 1 & 0  \end{pmatrix} \bullet f^\dagger A f 
=f^\dagger A f$. Put $x_1=f(e_1), x_2=f(e_2) \in B$, 
which are mutually orthogonal nonzero projections satisfying $x_1+ x_2 = 1_B$.
The regularity shows $Ax_1 + Ax_2=A1_B=dx_1 + dx_2$.
By $(f^\dagger A f)_{ij}=2\braket{e_i|f^\dagger A f |e_j}=2\braket{x_i| A x_j}$ and  
$\begin{pmatrix} 0 & 1  \\ 1 & 0  \end{pmatrix} \bullet f^\dagger A f  
=f^\dagger A f$, it follows that 
\begin{align}
\braket{x_1| A x_1} &= \braket{x_2| A x_2} = 0;
\\
\braket{x_1| A x_2} &= \braket{x_1+x_2| A x_2} = \psi(A x_2) = d\psi(x_2)
\\
&=\ol{\braket{x_2| A x_1}}=d\psi(x_1).
\end{align}
Thus $\psi(x_1)=\psi(x_2)=1/2$,
and $Ax_1= dx_2 + (dx_1-Ax_2)$ gives the orthogonal decomposition of $Ax_1$
 along $\C x_2 \oplus (x_2)^\perp$. This yields
\[
\frac{d^2}{2} = d^2 \psi(x_1) \geq \norm{Ax_1}_2^2
= \norm{d x_2}_2^2 + \norm{dx_1-Ax_2}_2^2
= \frac{d^2}{2} + \norm{dx_1-Ax_2}_2^2,
\]
hence $\norm{dx_1-Ax_2}_2=0$, i.e., $Ax_1=dx_2$ and $Ax_2=dx_1$.
Therefore we obtain $A(x_1-x_2) = -d(x_1-x_2)$, which shows 
$-d \in \Spec(A)$. 

\noindent $((2) \implies (1))$:
By Lemma \ref{lem:saevreal}, we can take a self-adjoint $x \in \ker(d \,\id +A)$
with $\norm{x}_2=1$. 
Decompose $x=x_+ - x_-$ into positive and negative parts $x_\pm \in B_+$,
Then we have 
\[
A x_+ - A x_-= Ax = -dx =dx_- - dx_+.
\]
 The self-adjointness of $A$ implies the orthogonality of eigenvectors
$\psi(x)=\braket{1 | x}=0$, i.e., $\psi(x_+)=\psi(x_-)$, 
hence the regularity implies 
$\psi(Ax_\pm)=d \psi(x_\pm)=d \psi(x_\mp)=\psi(d x_\mp)$. 
Note that the real quantum graph $A$ is CP; hence $Ax_\pm$ are positive.
Since $\psi$ is tracial and $x_\pm$ have disjoint supports, 
Lemma \ref{lem:tracialposneg} shows
\[
Ax_\pm = dx_\mp.
\]
 Thus $A(x_+ + x_-)=d(x_+ + x_-)$.
Since $\mathcal{G}$ is connected, we get $x_+ + x_- =c1_B$ for some $c>0$.
By $1=\norm{x}_2^2=\norm{x_+}_2^2 + \norm{x_-}_2^2=\norm{x_+ + x_-}_2^2 =c^2$,
we have $c=1, x_+ + x_- =1_B$.
Then $x_+ x_-=0$ shows $x_\pm^2=x_\pm(x_\pm + x_\mp)=x_\pm$,
hence $x_\pm$ are mutually orthogonal projections with $\psi(x_\pm)=1/2$.
Thus we obtain an injective $*$-homomorphism
$f: \C^2 \to B$ defined by $f(e_1)=x_+, f(e_2)=x_-$.
It satisfies
\begin{align}
2 \braket{e_i|f^\dagger Af | e_i} 
&= 2 \braket{x_\pm | A x_\pm} = 2d \braket{x_\pm | x_\mp} =0 \quad(i=1,2);
\\
2 \braket{e_1|f^\dagger Af | e_2} 
&= 2 \braket{x_+ | A x_-} = 2d \braket{x_+ | x_+} =d.
\end{align}
Thus $f^\dagger Af = \begin{pmatrix} 0 & d  \\ d & 0  \end{pmatrix}$,
which gives a surjective graph homomorphism $f^\op: \mathcal{G} \to K_2$.
\end{proof}

\begin{thm}\label{thm:bipartite}
Let $\mathcal{G} = (B,\psi,A)$ be a $d$-regular undirected tracial quantum graph. The following are equivalent:
\begin{enumerate}
\item[$(1)$]
$\mathcal{G}$ has a bipartite component.

\item[$(2)$]
$-d \in \Spec(A)$.
If $d=0$, we require that the multiplicity of $0 \in \Spec(A)$ is at least two, i.e., ${\dim B} \geq 2$.
\end{enumerate}
\end{thm}

\begin{proof} 
If $\dim B=1$, $\mathcal{G}\to K_2\sqcup T_1$ cannot be surjective to $K_2$. Hence $\mathcal{G}$ does not have a bipartite component, and $-d \in \Spec(A)$ does not hold as in the previous proof.
If $d=0$ and ${\dim B} \geq 2$, then $d=0=-d$ is the multiple root of $A=0$ and has a graph homomorphism $\mathcal{G}\to K_2\sqcup T_1$ that is surjective to $K_2$, hence $\mathcal{G}$ has a bipartite component and $-d \in \Spec(A)$.
In the sequel of the proof, we may assume $d>0$ and ${\dim B} \geq 2$. 

\noindent $((1) \implies (2))$:
We have a unital $*$-homomorphism $f: \C^3 \to B$ such that 
$\begin{pmatrix}
0 & 1 & 0 \\ 1 & 0 & 0 \\ 0 & 0 & 1
\end{pmatrix} \bullet f^\dagger A f =f^\dagger A f$ and $f$ is injective on $\C^2 \oplus 0$.  
Put $x_i=f(e_i) \in B$ for $i=1,2,3$, 
which are mutually orthogonal projections satisfying $x_1+x_2+x_3=1$ and $x_1,x_2$ are nonzero. 
Then the regularity implies 
\[
Ax_1 + A(1-x_1)=A1_B=dx_2 + d(1-x_2).
\]
By $(f^\dagger A f)_{ij}=3\braket{e_i|f^\dagger A f |e_j}=3\braket{x_i| A x_j}$ and  
$\begin{pmatrix}
0 & 1 & 0 \\ 1 & 0 & 0 \\ 0 & 0 & 1
\end{pmatrix} \bullet f^\dagger A f 
=f^\dagger A f$, it follows that 
\begin{align}
\braket{1-x_1| A x_2} &= \braket{1-x_2| A x_1} = \braket{1-x_3| A x_3} = 0;
\\
\braket{x_1| A x_2} &= \braket{x_1+(1-x_1)| A x_2} = \psi(A x_2) = d\psi(x_2)
\\
&=\ol{\braket{x_2| A x_1}}=d\psi(x_1).
\end{align}
Thus $\psi(x_1)=\psi(x_2)$,
and $Ax_1= dx_2 + (d(1-x_2)-A(1-x_1))$ gives the orthogonal decomposition of $Ax_1$
 along $\C x_2 \oplus (x_2)^\perp$. Then we have
\begin{align}
d^2 \psi(x_1) = \norm{A}^2 \norm{x_1}_2^2
&\geq \norm{Ax_1}_2^2
= \norm{d x_2}_2^2 + \norm{d(1-x_2)-A(1-x_1)}_2^2
\\
&= d^2 \psi(x_2) + \norm{d(1-x_2)-A(1-x_1)}_2^2,
\end{align}
hence $\norm{d(1-x_2)-A(1-x_1)}_2=0$, i.e., $A(1-x_1)=d(1-x_2)$ and $Ax_1=dx_2$.
By symmetry, we also have $Ax_2=dx_1$.
Therefore we obtain 
\[
A(x_1-x_2) = -d(x_1-x_2),
\]
 which shows $-d \in \Spec(A)$. 

\noindent $((2) \implies (1))$:
By Lemma \ref{lem:saevreal}, we can take a self-adjoint $x \in \ker(d\, \id +A)$. 
Consider the spectral projections $\{ p_\lambda | \lambda \in \Spec(x)\}$ of 
\[
x=\sum_{\lambda} \lambda p_\lambda
 = \sum_{\lambda>0} \lambda (p_\lambda - p_{-\lambda});
\quad x_+ = \sum_{\lambda>0} \lambda p_\lambda;
\quad x_- = \sum_{\lambda>0} \lambda p_{-\lambda}.
\]
In the same way as the proof of Theorem \ref{thm:connbipartite}, we obtain  $Ax_\pm=d x_\mp$ and $x_+ + x_- \in \ker (d\, \id - A)$.
Therefore it follows from Lemma \ref{lem:fcncalcev} that 
$p_\lambda + p_{-\lambda} \in \ker (d \,\id - A)$ for all $\lambda > 0$.
Thus it follows for a fixed $\lambda>0$ that
\begin{align}
A p_{\lambda} = d p_\lambda + d p_{-\lambda} - A p_{-\lambda} 
\leq d p_\lambda + d p_{-\lambda}.
\label{Apbound1}
\end{align}
Now $\lambda p_\lambda \leq x_+$ implies 
\begin{align}
A p_\lambda \leq \lambda^{-1} Ax_+ = \frac{d}{\lambda} x_- 
= \sum_{\mu>0} \frac{d \mu}{\lambda} p_{-\mu}.
\label{Apbound2}
\end{align}
By taking the meet of \eqref{Apbound1} and \eqref{Apbound2} in the lattice of self-adjoint elements in the commutative algebra $C^*(x)$, we obtain
\begin{align}
A p_\lambda \leq (d p_\lambda + d p_{-\lambda}) \wedge \sum_{\mu>0} \frac{d \mu}{\lambda} p_{-\mu} 
= d p_{-\lambda}.
\label{Ap<dp}
\end{align}
Similarly we have $A p_{-\lambda} \leq d p_{\lambda}$, i.e.,
\begin{align}
A p_\lambda = A(1 - p_{-\lambda}) \geq d(1 - p_{\lambda})= d p_{-\lambda}.
\label{Ap>dp}
\end{align}
Combining \eqref{Ap<dp} and \eqref{Ap>dp} we get $A p_\lambda = d p_{-\lambda}$ 
and $A p_{-\lambda} = d p_{\lambda}$.
Hence
$p_\lambda - p_{-\lambda} \in \ker(d \,\id +A)$ for all $\lambda > 0$.
Thus we may initially take $x=x_+ - x_- \in \ker(d \,\id +A)$ for mutually orthogonal nonzero projections $x_\pm \in B$ satisfying
\[
Ax_\pm = dx_\mp.
\]
Then we have a $*$-homomorphism
$f: \C^2 \oplus \C \to B$ defined by $f(e_1)=x_+, f(e_2)=x_-, f(e_3)=1-x_+ - x_-$
that is injective on $\C^2 \oplus 0$.
It satisfies
\begin{align}
3 \braket{e_i|f^\dagger Af | e_i} 
&= 3 \braket{x_\pm | A x_\pm} = 3d \braket{x_\pm | x_\mp} =0 \quad(i=1,2);
\\
3 \braket{e_1|f^\dagger Af | e_2} 
&= 3 \braket{x_+ | A x_-} = 3d \braket{x_+ | x_+} =3d \psi(x_+);
\\
3 \braket{e_i|f^\dagger Af | e_3} 
&= 3 \braket{x_\pm | d1- dx_+ - dx_-} = 0 \quad(i=1,2);
\\
3 \braket{e_3|f^\dagger Af | e_3} 
&= 3d \braket{1-x_+ - x_- | 1-x_+ - x_-} = 3d (1-2\psi(x_+)).
\end{align}
Thus $f^\dagger Af = \begin{pmatrix}
 0 & 3d \psi(x_+) & 0 \\
 3d \psi(x_+) & 0 & 0 \\
 0 & 0 & 3d (1-2\psi(x_+)) 
\end{pmatrix}$,
which gives a graph homomorphism $f^\op: \mathcal{G} \to K_2 \sqcup T_1$.
\end{proof}

\section{Two-colorability and bipartiteness}

It is known that a classical graph is bipartite if and only if it is two-colorable.
We compare the bipartiteness defined in this paper and the local two-colorability introduced in \cite{BGH2022quantum}.

\subsection{$t$-homomorphism}
The gap of bipartiteness and two-colorability arises from the two notions of graph homomorphisms: one is Definition \ref{dfn:adjhom} and the other is the following $t$-homomorphisms:  

\begin{dfn}[{Modified generalization of Brannan, Ganesan, Harris \cite{BGH2022quantum}}] \label{dfn:t-hom}
Let $\mathcal{G}_0 = 
(B_0,\psi_0,A_0,\mathcal{S}_0), \mathcal{G}_1 = (B_1,\psi_1,A_1,\mathcal{S}_1)$ 
be quantum graphs with $\delta_i$-forms $\psi_i$ and quantum relations 
$\mathcal{S}_i=\ran (A_i\bullet \cdot)\subset B(L^2(\mathcal{G}_i))$.
A $t$-homomorphism 
$(f,\mathcal{A}):\mathcal{G}_0 \overset{t}{\to} \mathcal{G}_1$ 
($t \in \{loc,q,qa,qc,C^*,alg\}$) is consisting of a unital $*$-homomorphism 
$f: B_1 \to B_0 \otimes \mathcal{A}$ and a unital $*$-algebra $\mathcal{A}$
satisfying
\begin{align}
f^\dagger (\mathcal{S}_0 \otimes 1_\mathcal{A}) f 
&\subset \mathcal{S}_1 \otimes \mathcal{A},	\label{t-hom}
\end{align}
where $f^\dagger \in B(L^2(\mathcal{G}_0),L^2(\mathcal{G}_1)) \otimes \mathcal{A}$ is the adjoint $(\cdot)^\dagger\otimes(\cdot)^*$ of $f$ as an operator in $B(L^2(\mathcal{G}_1),L^2(\mathcal{G}_0)) \otimes \mathcal{A}$,
and 
\begin{itemize}
\item
$\mathcal{A}=\C$ if $t=loc$ (local, classical); 
\item
$\mathcal{A}$ is finite-dimensional if $t=q$ (quantum);
\item
$\mathcal{A}=\mathcal{R}^\omega$ is the ultrapower of the hyperfinite 
$II_1$-factor $\mathcal{R}$ by a free ultrafilter $\omega$ on $\N$ if $t=qa$ (quantum approximate);
\item
$\mathcal{A}$ is a tracial $C^*$-algebra if $t=qc$ (quantum commuting);
\item
$\mathcal{A}$ is a $C^*$-algebra if $t=C^*$;
\item
$\mathcal{A}$ is a unital $*$-algebra if $t=alg$.
\end{itemize}
These notions of $t$ show what kind of quantum correlation is allowed in the corresponding graph homomorphism game.

We say that a $t$-homomorphism 
$(f,\mathcal{A}):\mathcal{G}_0 \overset{t}{\to} \mathcal{G}_1$ 
is:
\begin{itemize}
\item (vertex-)surjective if $f: B_1 \to B_0 \otimes \mathcal{A}$ is injective.
\end{itemize}
\end{dfn}

This definition means that the pushforward of the edges of $\mathcal{G}_0$ by the mapping $(f,\mathcal{A})$ are edges of $\mathcal{G}_1$.

\begin{rmk}
For a $t$-homomorphism 
$(f,\mathcal{A}):\mathcal{G}_0 \overset{t}{\to} \mathcal{G}_1$,
the best definition of (vertex-)injectivity is not sure. 
If it is a classical homomorphism between classical graphs, then 
$(f,\mathcal{A}=\C)$ is injective if and only if 
$(\delta_0/\delta_1)f$ is a coisometry 
$f f^\dagger = (\delta_1/\delta_0)^2 \id_{B_1} \otimes 1_\mathcal{A}$,
so this is a candidate for the definition of injectivity.
Another weaker candidate is the injectivity of $f^\dagger: B_0 \to B_1 \otimes \mathcal{A}$.

On the other hand, in classical case $(f,\C)$ is surjective if and only if
$f^\dagger f \geq (\delta_1/\delta_0)^2 \id_{B_0} \otimes 1_\mathcal{A}$,
 which may be too strong for the definition of surjectivity of general
$(f,\mathcal{A})$.

Consider a toy model $f: \C^4 \to \C^2 \otimes M_2$ of a quantum $4$-coloring of $2$ vertices $(f,M_2):(\C^2,\tau,0)\overset{q}{\to}K_4$ given by
\[
f(e_1)=e_1 \otimes e_{11}; \quad
f(e_2)=e_2 \otimes e_{11}; \quad
f(e_3)=e_1 \otimes e_{22}; \quad
f(e_4)=e_2 \otimes e_{22}.
\]
Then coisometry condition $ff^\dagger=2\, \id_{\C^2} \otimes 1_{M_2}$ holds, hence $(f,M_2)$ is injective in the strong sense. On the other hand, we have
an injective homomorphism $f$ but
$f^\dagger f=2 [(e_1+e_2)\otimes e_{11} + (e_3+e_4)\otimes e_{22} ]
\not\geq 2\ \id_{\C^4} \otimes 1_{M_2}$, hence $(f,M_2)$ is surjective only in the weak sense as defined above.
\end{rmk}

\begin{notation*}
For a quantum graph $\mathcal{G}=(B.\psi,A)$ and a unital algebra $\mathcal{A}$, we abbreviate by $A\bullet \cdot$ the left Schur product by $A$ acting on the first tensor component of $B(L^2(\mathcal{G})) \otimes \mathcal{A}$.

If we faithfully represent $\mathcal{A}\subset B(H)$ on a Hilbert space $H$, 
we may regard
$f: B_1 \to B_0 \otimes \mathcal{A}$ as 
$f: L^2(\mathcal{G}_1)\otimes H \to H \otimes L^2(\mathcal{G}_0)
 \in B(L^2(\mathcal{G}_1),L^2(\mathcal{G}_0)) \otimes B(H)$.
By this identification, we denote $f$ in string diagrams by
\[
f=
\begin{sd}
\node[circle,draw] (f) at (0,0){$f$};
\draw (-\sep/2,-\sep/2)node[left]{$B_1$}--(f)
	--(\sep/2,\sep/2)node[right]{$B_0$};
\draw[->-] (\sep/2,-\sep/2)node[right]{$H$}--(f);
\draw[->-] (f)--(-\sep/2,\sep/2)node[left]{$H$};
\end{sd},
\]
where $H$ is drawn as oriented strings. 
Even if $\mathcal{A}$ is not a $C^*$-algebra, such a diagram formally makes sense by thinking of the string of $H$ as an indicator of the order of multiplication in $\mathcal{A}$.

Note that $f$ is unital; multiplicative; $*$-preserving (real) respectively if and only if the following are satisfied:
\begin{align}
\begin{sd}
\node[circle,draw] (f) at (0,0){$f$};
\draw (-\sep/4,-\sep/4) arc(45:-315:\unode)
	--(f)--(\sep/2,\sep/2);
\draw[->-] (\sep/2,-\sep/2)--(f);
\draw[->-] (f)--(-\sep/2,\sep/2);
\end{sd}
=
\begin{sd}
\draw (\sep/4,\sep/4) arc(45:-315:\unode)
	--(\sep/2,\sep/2);
\draw[->-] (\sep/2,-\sep/2)--(-\sep/2,\sep/2);
\end{sd}
;
\begin{sd}
\node[circle,draw] (f1) at (-\sep/4,\sep/6){$f$};
\node[circle,draw] (f2) at (\sep/4,-\sep/6){$f$};
\draw (-\sep*3/4,-\sep/2)--(f1) (0,-\sep/2)--(f2)
 	(f1)to[out=45,in=45] coordinate[midway](m) (f2)
	(m)--(\sep/2,\sep/2);
\draw[->-] (\sep*3/4,-\sep/2)--(f2);
\draw (f2)--(f1);
\draw[->-] (f1)--(-\sep*3/4,\sep/2);
\end{sd}
=
\begin{sd}
\node[circle,draw] (f) at (0,0){$f$};
\draw (-\sep/4,-\sep/4)coordinate(m) --(f)
	--(\sep/2,\sep/2);
\draw (-\sep*3/4,-\sep/2)to[out=45,in=180] (m)
	to[out=0,in=45](0,-\sep/2);
\draw[->-] (\sep/2,-\sep/2)--(f);
\draw[->-] (f)--(-\sep/2,\sep/2);
\end{sd}
;
\begin{sd}
\node[circle,draw] (f) at (0,0){$f^\dagger$};
\draw (-\sep/2,-\sep/2)to[out=90,in=135](f)
	to[out=-45,in=-90](\sep/2,\sep/2);
\draw[->-] (-\sep/4,-\sep/2)--(f);
\draw[->-] (f)--(\sep/4,\sep/2);
\end{sd}
=
\begin{sd}
\node[circle,draw] (f) at (0,0){$f$};
\draw (-\sep/2,-\sep/2)--(f)
	--(\sep/2,\sep/2);
\draw[->-] (\sep/2,-\sep/2)--(f);
\draw[->-] (f)--(-\sep/2,\sep/2);
\end{sd}.
	\label{f*-hom}
\end{align}
\end{notation*}

\begin{prop}\label{prop:t-homch}
Let $(f,\mathcal{A}):\mathcal{G}_0 \overset{t}{\to} \mathcal{G}_1$ be as in Definition \ref{dfn:t-hom} without assumption \eqref{t-hom}.
The following are equivalent:
\begin{description}
\item[$(1)$]
The inclusion \eqref{t-hom}: 
$f^\dagger (\mathcal{S}_0 \otimes 1_\mathcal{A}) f 
\subset \mathcal{S}_1 \otimes \mathcal{A}$;
\item[$(2)$] 
$A_1 \bullet (f^\dagger (A_0 \bullet T\otimes 1_\mathcal{A}) f )
=f^\dagger (A_0 \bullet T\otimes 1_\mathcal{A}) f$ 
for any $T \in B(L^2(\mathcal{G}_0))$;
\item[$(3)$]
$\braket{S | f^\dagger (T\otimes 1_\mathcal{A}) f}
=0$ in $\mathcal{A}$
for any $S \in \mathcal{S}_1^\perp$ and $T \in \mathcal{S}_0$, where 
$\braket{S | \cdot}=\Psi_1(S^\dagger \cdot)
=\delta_1^2 \braket{1_{B_1}|S^* \bullet \cdot|1_{B_1}}_{\psi_1}$
 as in \eqref{innerprodonB(L2)} acts on the first tensor component;
\item[$(4)$]
The (adjoint of) diagrammatic definition of quantum graph homomorphism by \cite[Definition 5.4]{Musto2018compositional}:
\begin{align}
\begin{sd}
\path coordinate (m0*) at (\sep*2/3,\sep/2) 
	  coordinate (m1) at ($-1*(m0*)$)
	  coordinate (m0) at (-\sep/3,\sep*5/6)
	  coordinate (m1*) at ($-1*(m0)$)
	  coordinate (a1) at (-\sep/3,\sep)
	  coordinate (b2) at ($-1*(a1)$)
	  coordinate (a2) at (\sep,\sep)
	  coordinate (b1) at ($-1*(a2)$)
	  coordinate (ha) at (-\sep,\sep)
	  coordinate (hb) at ($-1*(ha)$);
\draw[name path=rightB] (b1) to[out=90,in=180] (m1) 
	to[out=0,in=180] node[circle,draw,fill=white](A1){$A_1$} (m1*)
	to[out=0,in=-90] (m0*);
\draw[name path=leftB] (m1) to[out=90,in=180] (m0) 
	to[out=0,in=180] node[circle,draw,fill=white](A0){$A_0$} (m0*)
	to[out=0,in=-90] (a2);
\draw (b2)to[out=90,in=-90] (m1*) (m0)to[out=90,in=-90] (a1);
\draw[->-, name path=H] (hb) to[out=90,in=-90] (ha);
\path [name intersections={of= H and rightB, by= f1}];
\path [name intersections={of= H and leftB, by= f2}];
\path node[draw,circle,fill=white] at (f1) {$f$} 
	node[draw,circle,fill=white] at (f2) {$f$};
\end{sd}
=
\delta_1^2 
\begin{sd}
\path coordinate (m0*) at (\sep*2/3,\sep/2) 
	  coordinate (m1) at ($-1*(m0*)$)
	  coordinate (m0) at (-\sep/3,\sep*5/6)
	  coordinate (m1*) at ($-1*(m0)$)
	  coordinate (a1) at (-\sep/3,\sep)
	  coordinate (b2) at ($-1*(a1)$)
	  coordinate (a2) at (\sep,\sep)
	  coordinate (b1) at ($-1*(a2)$)
	  coordinate (ha) at (-\sep,\sep)
	  coordinate (hb) at ($-1*(ha)$);
\draw[name path=rightB] (b2) 
	to[out=90,in=-90] (m0*);
\draw[name path=leftB] (b1) ++(\sep/4,0)--++(0,\sep/2)
	to[out=90,in=180] (m0) 
	to[out=0,in=180] node[circle,draw,fill=white](A0){$A_0$} (m0*)
	to[out=0,in=-90] (a2);
\draw  (m0)to[out=90,in=-90] (a1);
\draw[->-, name path=H] (hb) to[out=90,in=-90] (ha);
\path [name intersections={of= H and rightB, by= f1}];
\path [name intersections={of= H and leftB, by= f2}];
\path node[draw,circle,fill=white] at (f1) {$f$} 
	node[draw,circle,fill=white] at (f2) {$f$};
\end{sd}.
\end{align}
\end{description}
\end{prop}

\begin{proof}
$((1) \iff (2))$: Since $\mathcal{S}_1 \otimes \mathcal{A}
=\ran(A_1 \bullet \cdot) \otimes \mathcal{A}$ and $A_1 \bullet \cdot$ is a projection, (1) means that $f^\dagger (S\otimes 1_\mathcal{A}) f$ is invariant under the action of $A_1 \bullet \cdot$ for all $S \in \mathcal{S}_0$. 
Thus  (1) is equivalent to
\begin{align}
A_1 \bullet (f^\dagger (A_0 \bullet T\otimes 1_\mathcal{A}) f )
=f^\dagger (A_0 \bullet T \otimes 1_\mathcal{A}) f	
	\qquad \forall T \in B(L^2(\mathcal{G}_0)).
\end{align}

\noindent $((1)\iff(3))$: Note that 
\[
(\mathcal{S}_1^\perp)^\perp \otimes \mathcal{A}
=\{ X \in B(L^2(\mathcal{G}_1))\otimes \mathcal{A} \vert \braket{S|X}=0 \ \forall S \in \mathcal{S}_1^\perp \}.
\] 
Indeed $\subset$ is obvious and $\supset$ is shown by choosing a presentation $X=\sum_j T_j \otimes a_j \in \textrm{(RHS)}$ 
with independent $a_j$'s in $\mathcal{A}$ 
to deduce $\braket{S|T_j}=0$ 
from $\braket{S|X} = \sum_j \braket{S|T_j}  a_j=0$.
Thus (1) is equivalent to (3).

\noindent $((2)\implies(4))$: In string diagrams, (2) is expressed as follows:
\begin{align}
\delta_1^{-2} \delta_0^{-2}
\begin{sd}
\path coordinate (m0) at (\sep/3,\sep/2)
	  coordinate (m0*) at (\sep/3,-\sep/2) 
	  coordinate (m1) at (-\sep/2,\sep*5/6)
	  coordinate (m1*) at (-\sep/2,-\sep*5/6)
	  coordinate (a1) at (-\sep/2,\sep*7/6)
	  coordinate (b1) at (-\sep/2,-\sep*7/6)
	  coordinate (a2) at (\sep,\sep*7/6)
	  coordinate (b2) at (\sep,-\sep*7/6)
	  coordinate (ha) at (\sep/8,\sep*7/6)
	  coordinate (hm) at (-\sep/2,0)
	  coordinate (hb) at (\sep/8,-\sep*7/6);
\draw (m0*)	to[out=180,in=180] node[circle,draw,fill=white](A0){$A_0$} (m0);\draw (m0*)	to[out=0,in=0] node[circle,draw,fill=white](T){$T$} (m0);
\draw (m1*)	to[out=180,in=180] node[circle,draw,fill=white](A1){$A_1$} (m1);
\draw[name path=rightB] (m1*) to[out=0,in=-90] (m0*) ;
\draw[name path=leftB] (m0) to[out=90,in=0] (m1) ;
\draw (b1)to[out=90,in=-90] (m1*) (m1)to[out=90,in=-90] (a1);
\draw[->-, name path=H] (hb) to[out=90,in=-90] (hm)
	--(hm) to[out=90,in=-90] (ha);
\path [name intersections={of= H and rightB, by= f1}];
\path [name intersections={of= H and leftB, by= f2}];
\path node[draw,circle,fill=white] at (f1) {$f$} 
	node[draw,circle,fill=white] at (f2) {$f^\dagger$};
\end{sd}
=
\delta_0^{-2} 
\begin{sd}
\path coordinate (m0) at (\sep/3,\sep/2)
	  coordinate (m0*) at (\sep/3,-\sep/2) 
	  coordinate (m1) at (-\sep/2,\sep*5/6)
	  coordinate (m1*) at (-\sep/2,-\sep*5/6)
	  coordinate (a1) at (-\sep/2,\sep*7/6)
	  coordinate (b1) at (-\sep/2,-\sep*7/6)
	  coordinate (a2) at (\sep,\sep*7/6)
	  coordinate (b2) at (\sep,-\sep*7/6)
	  coordinate (ha) at (\sep/8,\sep*7/6)
	  coordinate (hm) at (-\sep/2,0)
	  coordinate (hb) at (\sep/8,-\sep*7/6);
\draw (m0*)	to[out=180,in=180] node[circle,draw,fill=white](A0){$A_0$} (m0);\draw (m0*)	to[out=0,in=0] node[circle,draw,fill=white](T){$T$} (m0);
\draw[name path=rightB] (b1) to[out=90,in=-90] (m0*) ;
\draw[name path=leftB] (m0) to[out=90,in=-90] (a1) ;
\draw[->-, name path=H] (hb) to[out=90,in=-90] (hm)
	--(hm) to[out=90,in=-90] (ha);
\path [name intersections={of= H and rightB, by= f1}];
\path [name intersections={of= H and leftB, by= f2}];
\path node[draw,circle,fill=white] at (f1) {$f$} 
	node[draw,circle,fill=white] at (f2) {$f^\dagger$};
\end{sd}.
\end{align}
Since $T \in B(L^2(\mathcal{G}_0)) \cong B_0 \otimes B_0^*$ is arbitrary, we may replace 
$T=
\begin{sd}
\node[circle,draw,fill=white] (T) at (0,0) {$T$};
\draw (T.north)--++(0,\sep/8) (T.south)--++(0,-\sep/8);
\end{sd}$ with open ends 
$\begin{sd}
\node[draw,dashed] (T) at (0,0) {};
\draw (T.north)--++(0,\sep/8) (T.south)--++(0,-\sep/8);
\end{sd}$ of strings of $B_0$. 
We move the open ends to the top 
(insert $\begin{sd}
\node[draw,dashed,below] (T) at (0,0) {};
\draw (T.north)to[out=-90,in=-90](\sep/2,0)
	(T.south)to[out=90,in=-90](\sep*3/4,0);
\end{sd}$)
and move the top left string of $B_1$ to the bottom
(postcompose the right end of $\begin{sd}
\draw (0,0)to[out=90,in=90](\sep/2,0);
\end{sd}$) 
to obtain 
\begin{align}
\begin{sd}
\path coordinate (m0) at (\sep/3,\sep/2)
	  coordinate (m0*) at (\sep/2,-\sep/3) 
	  coordinate (m1) at (-\sep*2/3,\sep*2/3)
	  coordinate (m1*) at (-\sep/2,-\sep*5/6)
	  coordinate (r1) at (-\sep/2,\sep)
	  coordinate (b1) at (-\sep,-\sep*7/6)
	  coordinate (b2) at (-\sep/2,-\sep*7/6)
	  coordinate (a1) at (\sep/2,\sep*7/6)
	  coordinate (a2) at (\sep,\sep*7/6)
	  coordinate (ha) at (0,\sep*7/6)
	  coordinate (hm) at (-\sep/8,0)
	  coordinate (hb) at (0,-\sep*7/6);
\draw (m0*)	to[out=180,in=-90] node[circle,draw,fill=white](A0){$A_0$} (m0);\draw (m0*)	to[out=0,in=-90] (a2) (m0)to[out=0,in=-90] (a1);
\draw (m1*)	to[out=180,in=0] node[circle,draw,fill=white](A1){$A_1$} (m1);
\draw[name path=rightB] (m1*) to[out=0,in=-90] (m0*) ;
\draw[name path=leftB] (m0) to[out=180,in=0] (r1) to[out=180,in=90] (m1);
\draw (b2)to[out=90,in=-90] (m1*) (m1)to[out=180,in=90] (b1);
\draw[->-, name path=H] (hb) to[out=90,in=-90] (hm)
	--(hm) to[out=90,in=-90] (ha);
\path [name intersections={of= H and rightB, by= f1}];
\path [name intersections={of= H and leftB, by= f2}];
\path node[draw,circle,fill=white] at (f1) {$f$} 
	node[draw,circle,fill=white] at (f2) {$f^\dagger$};
\end{sd}
=
\delta_1^{2} 
\begin{sd}
\path coordinate (m0) at (\sep/3,\sep/2)
	  coordinate (m0*) at (\sep/2,-\sep/3) 
	  coordinate (m1) at (-\sep*2/3,\sep*2/3)
	  coordinate (m1*) at (-\sep/2,-\sep*5/6)
	  coordinate (r1) at (-\sep/3,\sep*2/3)
	  coordinate (b1) at (-\sep,-\sep*7/6)
	  coordinate (b2) at (-\sep/2,-\sep*7/6)
	  coordinate (a1) at (\sep/2,\sep*7/6)
	  coordinate (a2) at (\sep,\sep*7/6)
	  coordinate (ha) at (0,\sep*7/6)
	  coordinate (hm) at (-\sep/8,0)
	  coordinate (hb) at (0,-\sep*7/6);
\draw (m0*)	to[out=180,in=-90] node[circle,draw,fill=white](A0){$A_0$} (m0);\draw (m0*)	to[out=0,in=-90] (a2) (m0)to[out=0,in=-90] (a1);
\draw[name path=rightB] (b2) to[out=90,in=-90] (m0*) ;
\draw[name path=leftB] (m0) to[out=180,in=0] (r1) to[out=180,in=90] (b1);
\draw[->-, name path=H] (hb) to[out=90,in=-90] (hm)
	--(hm) to[out=90,in=-90] (ha);
\path [name intersections={of= H and rightB, by= f1}];
\path [name intersections={of= H and leftB, by= f2}];
\path node[draw,circle,fill=white] at (f1) {$f$} 
	node[draw,circle,fill=white] at (f2) {$f^\dagger$};
\end{sd}.
\end{align}
Therefore it follows by the realness \eqref{f*-hom} of $f$ that
\begin{align}
\begin{sd}
\path coordinate (m0*) at (\sep*2/3,\sep/2) 
	  coordinate (m1) at ($-1*(m0*)$)
	  coordinate (m0) at (-\sep/3,\sep*5/6)
	  coordinate (m1*) at ($-1*(m0)$)
	  coordinate (a1) at (-\sep/3,\sep)
	  coordinate (b2) at ($-1*(a1)$)
	  coordinate (a2) at (\sep,\sep)
	  coordinate (b1) at ($-1*(a2)$)
	  coordinate (ha) at (-\sep,\sep)
	  coordinate (hb) at ($-1*(ha)$);
\draw[name path=rightB] (b1) to[out=90,in=180] (m1) 
	to[out=0,in=180] node[circle,draw,fill=white](A1){$A_1$} (m1*)
	to[out=0,in=-90] (m0*);
\draw[name path=leftB] (m1) to[out=90,in=180] (m0) 
	to[out=0,in=180] node[circle,draw,fill=white](A0){$A_0$} (m0*)
	to[out=0,in=-90] (a2);
\draw (b2)to[out=90,in=-90] (m1*) (m0)to[out=90,in=-90] (a1);
\draw[->-, name path=H] (hb) to[out=90,in=-90] (ha);
\path [name intersections={of= H and rightB, by= f1}];
\path [name intersections={of= H and leftB, by= f2}];
\path node[draw,circle,fill=white] at (f1) {$f$} 
	node[draw,circle,fill=white] at (f2) {$f$};
\end{sd}
=
\delta_1^2 
\begin{sd}
\path coordinate (m0*) at (\sep*2/3,\sep/2) 
	  coordinate (m1) at ($-1*(m0*)$)
	  coordinate (m0) at (-\sep/3,\sep*5/6)
	  coordinate (m1*) at ($-1*(m0)$)
	  coordinate (a1) at (-\sep/3,\sep)
	  coordinate (b2) at ($-1*(a1)$)
	  coordinate (a2) at (\sep,\sep)
	  coordinate (b1) at ($-1*(a2)$)
	  coordinate (ha) at (-\sep,\sep)
	  coordinate (hb) at ($-1*(ha)$);
\draw[name path=rightB] (b2) 
	to[out=90,in=-90] (m0*);
\draw[name path=leftB] (b1) ++(\sep/4,0)--++(0,\sep/2)
	to[out=90,in=180] (m0) 
	to[out=0,in=180] node[circle,draw,fill=white](A0){$A_0$} (m0*)
	to[out=0,in=-90] (a2);
\draw  (m0)to[out=90,in=-90] (a1);
\draw[->-, name path=H] (hb) to[out=90,in=-90] (ha);
\path [name intersections={of= H and rightB, by= f1}];
\path [name intersections={of= H and leftB, by= f2}];
\path node[draw,circle,fill=white] at (f1) {$f$} 
	node[draw,circle,fill=white] at (f2) {$f$};
\end{sd}.
\end{align}
$((4)\implies(2))$: We can transform the diagrams conversely to go back from (4) to (2).
\end{proof}

Note that \cite{BGH2022quantum} defined the quantum-to-classical $t$-homomorphisms by the following conditions instead of \eqref{t-hom}
to omit self-loops in particular for the coloring problem.
\begin{align}
f^\dagger (\mathcal{S}_0 \cap \mathcal{S}_{T(B_0,\psi_0)}^\perp \otimes 1_\mathcal{A}) f 
&\subset \mathcal{S}_1 \otimes \mathcal{A};	\label{t-hom1}
\\
f^\dagger (\mathcal{S}_{T(B_0,\psi_0)} \otimes 1_\mathcal{A}) f  
&\subset \mathcal{S}_{T(B_1,\psi_1)} \otimes \mathcal{A}.   \label{t-hom2}
\end{align}
\eqref{t-hom} and \label{t-hom1} coincide under some assumptions, and
as a generalization of \cite[Lemma 4.8]{BGH2022quantum}, the second condition \eqref{t-hom2} is redundant as shown below.

\begin{lem}
Let $(f,\mathcal{A}):\mathcal{G}_0 \overset{t}{\to} \mathcal{G}_1$ be as in Definition \ref{dfn:t-hom} without assumption \eqref{t-hom}.
\begin{description}
\item[$(1)$]
The inclusion \eqref{t-hom2} always holds. 
\item[$(2)$]
\eqref{t-hom} is equivalent to \eqref{t-hom1} if $\mathcal{G}_0$ is irreflexive, or if $\mathcal{G}_0$ has no partial loops and $\mathcal{G}_1$ is reflexive.
\end{description}
\end{lem}

\begin{proof}
$(1)$ Recall that the adjacency matrix of the trivial graph $T(B_i,\psi_i)$ is $\id_{B_i}$. 
Thus \eqref{t-hom2} is equivalent to
\begin{align}
\begin{sd}
\path coordinate (m0*) at (\sep*2/3,\sep/2) 
	  coordinate (m1) at ($-1*(m0*)$)
	  coordinate (m0) at (-\sep/3,\sep*5/6)
	  coordinate (m1*) at ($-1*(m0)$)
	  coordinate (a1) at (-\sep/3,\sep)
	  coordinate (b2) at ($-1*(a1)$)
	  coordinate (a2) at (\sep,\sep)
	  coordinate (b1) at ($-1*(a2)$)
	  coordinate (ha) at (-\sep,\sep)
	  coordinate (hb) at ($-1*(ha)$);
\draw[name path=rightB] (b1) to[out=90,in=180] (m1) 
	to[out=0,in=180]  (m1*)
	to[out=0,in=-90] (m0*);
\draw[name path=leftB] (m1) to[out=90,in=180] (m0) 
	to[out=0,in=180]  (m0*)
	to[out=0,in=-90] (a2);
\draw (b2)to[out=90,in=-90] (m1*) (m0)to[out=90,in=-90] (a1);
\draw[->-, name path=H] (hb) to[out=90,in=-90] (ha);
\path [name intersections={of= H and rightB, by= f1}];
\path [name intersections={of= H and leftB, by= f2}];
\path node[draw,circle,fill=white] at (f1) {$f$} 
	node[draw,circle,fill=white] at (f2) {$f$};
\end{sd}
&=
\delta_1^2 
\begin{sd}
\path coordinate (m0*) at (\sep*2/3,\sep/2) 
	  coordinate (m1) at ($-1*(m0*)$)
	  coordinate (m0) at (-\sep/3,\sep*5/6)
	  coordinate (m1*) at ($-1*(m0)$)
	  coordinate (a1) at (-\sep/3,\sep)
	  coordinate (b2) at ($-1*(a1)$)
	  coordinate (a2) at (\sep,\sep)
	  coordinate (b1) at ($-1*(a2)$)
	  coordinate (ha) at (-\sep,\sep)
	  coordinate (hb) at ($-1*(ha)$);
\draw[name path=rightB] (b2) 
	to[out=90,in=-90] (m0*);
\draw[name path=leftB] (b1) ++(\sep/4,0)--++(0,\sep/2)
	to[out=90,in=180] (m0) 
	to[out=0,in=180]  (m0*)
	to[out=0,in=-90] (a2);
\draw  (m0)to[out=90,in=-90] (a1);
\draw[->-, name path=H] (hb) to[out=90,in=-90] (ha);
\path [name intersections={of= H and rightB, by= f1}];
\path [name intersections={of= H and leftB, by= f2}];
\path node[draw,circle,fill=white] at (f1) {$f$} 
	node[draw,circle,fill=white] at (f2) {$f$};
\end{sd}.
\end{align}
This is proved by the multiplicativity \eqref{f*-hom} of $f$ and Frobenius equality:
\begin{align}
\begin{sd}
\path coordinate (m0*) at (\sep*2/3,\sep/2) 
	  coordinate (m1) at ($-1*(m0*)$)
	  coordinate (m0) at (-\sep/3,\sep*5/6)
	  coordinate (m1*) at ($-1*(m0)$)
	  coordinate (a1) at (-\sep/3,\sep)
	  coordinate (b2) at ($-1*(a1)$)
	  coordinate (a2) at (\sep,\sep)
	  coordinate (b1) at ($-1*(a2)$)
	  coordinate (ha) at (-\sep,\sep)
	  coordinate (hb) at ($-1*(ha)$);
\draw[name path=rightB] (b1) to[out=90,in=180] (m1) 
	to[out=0,in=180]  (m1*)
	to[out=0,in=-90] (m0*);
\draw[name path=leftB] (m1) to[out=90,in=180] (m0) 
	to[out=0,in=180]  (m0*)
	to[out=0,in=-90] (a2);
\draw (b2)to[out=90,in=-90] (m1*) (m0)to[out=90,in=-90] (a1);
\draw[->-, name path=H] (hb) to[out=90,in=-90] (ha);
\path [name intersections={of= H and rightB, by= f1}];
\path [name intersections={of= H and leftB, by= f2}];
\path node[draw,circle,fill=white] at (f1) {$f$} 
	node[draw,circle,fill=white] at (f2) {$f$};
\end{sd}
&=
\begin{sd}
\path coordinate (m0*) at (\sep/4,\sep*2/3) 
	  coordinate (m1) at ($-1*(m0*)$)
	  coordinate (m0) at (\sep/4,\sep/3)
	  coordinate (m1*) at ($-1*(m0)$)
	  coordinate (a1) at (-\sep/3,\sep)
	  coordinate (b2) at ($-1*(a1)$)
	  coordinate (a2) at (\sep,\sep)
	  coordinate (b1) at ($-1*(a2)$)
	  coordinate (ha) at (-\sep,\sep)
	  coordinate (Bl) at (-\sep*2/3,0)
	  coordinate (Br) at (\sep*2/3,0)
 	  coordinate (hb) at ($-1*(ha)$);
\draw[name path=rightB] (m1*)	to[out=0,in=-90] (Br) to[out=90,in=0] (m0);
\draw[name path=leftB] (m1*) to[out=180,in=-90] (Bl) to[out=90,in=180] (m0); 
\draw (a1) to[out=-90,in=180]  (m0*)
	to[out=0,in=-90] (a2)
	(m0*)to (m0);
\draw (b1) to[out=90,in=180] (m1) 
	to[out=0,in=90]  (b2)
	(m1)to (m1*);
\draw[->-, name path=H] (hb) to[out=90,in=-90] (ha);
\path [name intersections={of= H and rightB, by= f1}];
\path [name intersections={of= H and leftB, by= f2}];
\path node[draw,circle,fill=white] at (f1) {$f$} 
	node[draw,circle,fill=white] at (f2) {$f$};
\end{sd}
\overset{\eqref{f*-hom}}{=}
\begin{sd}
\path coordinate (m0*) at (\sep/4,\sep*2/3) 
	  coordinate (m1) at (-\sep/4,-\sep*3/4)
	  coordinate (m0) at (-\sep/4,-\sep/8)
	  coordinate (m1*) at (-\sep/4,-\sep/2)
	  coordinate (a1) at (-\sep/3,\sep)
	  coordinate (b2) at ($-1*(a1)$)
	  coordinate (a2) at (\sep,\sep)
	  coordinate (b1) at ($-1*(a2)$)
	  coordinate (ha) at (-\sep,\sep)
 	  coordinate (hb) at ($-1*(ha)$);
\draw[name path=B] (m0) to[out=90,in=-90] (m0*);
\draw (m1*) to[out=180,in=180] (m0) to[out=0,in=0] (m1*); 
\draw (a1) to[out=-90,in=180]  (m0*)
	to[out=0,in=-90] (a2);
\draw (b1) to[out=90,in=180] (m1) 
	to[out=0,in=90]  (b2)
	(m1)to (m1*);
\draw[->-, name path=H] (hb)--++(0,\sep/2) to[out=90,in=-90] (ha);
\path [name intersections={of= H and B, by= f1}];
\path node[draw,circle,fill=white] at (f1) {$f$} ;
\end{sd}
\\
&\overset{(\delta_1\text{-form})}{=}
\delta_1^2
\begin{sd}
\path coordinate (m0*) at (\sep/4,\sep*2/3) 
	  coordinate (m1) at (-\sep/4,-\sep*2/3)
	  coordinate (m0) at (-\sep/4,-\sep/2)
	  coordinate (m1*) at (-\sep/4,-\sep/2)
	  coordinate (a1) at (-\sep/3,\sep)
	  coordinate (b2) at ($-1*(a1)$)
	  coordinate (a2) at (\sep,\sep)
	  coordinate (b1) at ($-1*(a2)$)
	  coordinate (ha) at (-\sep,\sep)
 	  coordinate (hb) at ($-1*(ha)$);
\draw[name path=B] (m1) to[out=90,in=-90] (m0*);
\draw (a1) to[out=-90,in=180]  (m0*)
	to[out=0,in=-90] (a2);
\draw (b1) to[out=90,in=180] (m1) 
	to[out=0,in=90]  (b2);
\draw[->-, name path=H] (hb) to[out=90,in=-90] (ha);
\path [name intersections={of= H and B, by= f1}];
\path node[draw,circle,fill=white] at (f1) {$f$} ;
\end{sd}
\overset{\eqref{f*-hom}}{=}
\delta_1^2 
\begin{sd}
\path coordinate (m0*) at (\sep*2/3,\sep/2) 
	  coordinate (m1) at ($-1*(m0*)$)
	  coordinate (m0) at (-\sep/3,\sep*5/6)
	  coordinate (m1*) at ($-1*(m0)$)
	  coordinate (a1) at (-\sep/3,\sep)
	  coordinate (b2) at ($-1*(a1)$)
	  coordinate (a2) at (\sep,\sep)
	  coordinate (b1) at ($-1*(a2)$)
	  coordinate (ha) at (-\sep,\sep)
	  coordinate (hb) at ($-1*(ha)$);
\draw[name path=rightB] (b2) 
	to[out=90,in=-90] (m0*);
\draw[name path=leftB] (b1) ++(\sep/4,0)--++(0,\sep/2)
	to[out=90,in=180] (m0) 
	to[out=0,in=180]  (m0*)
	to[out=0,in=-90] (a2);
\draw  (m0)to[out=90,in=-90] (a1);
\draw[->-, name path=H] (hb) to[out=90,in=-90] (ha);
\path [name intersections={of= H and rightB, by= f1}];
\path [name intersections={of= H and leftB, by= f2}];
\path node[draw,circle,fill=white] at (f1) {$f$} 
	node[draw,circle,fill=white] at (f2) {$f$};
\end{sd}.
\end{align}
\noindent $(2)$
(i) If $\mathcal{G}_0$ is irreflexive, then 
$\mathcal{S}_0 \subset 
\mathcal{S}_{T(B_0,\psi_0)}^\perp=\mathcal{S}_{K(B_0,\psi_0)}$. 
Thus \eqref{t-hom1} is exactly equal to \eqref{t-hom}.
(ii) Note that \eqref{t-hom} always implies \eqref{t-hom1} by the trivial inclusion
\[
f^\dagger(\mathcal{S}_0 \cap \mathcal{S}_{T(B_0,\psi_0)}^\perp 
	\otimes 1_\mathcal{A})f
\subset f^\dagger(\mathcal{S}_0 \otimes 1_\mathcal{A})f
\overset{\eqref{t-hom}}{\subset} \mathcal{S}_1 \otimes \mathcal{A}.
\]
If $\mathcal{G}_1$ is reflexive, then 
$\mathcal{S}_{T(B_1,\psi_1)} \subset \mathcal{S}_1$,
and no partial loops means that 
$\mathcal{S}_0=\mathcal{S}_0 \cap \mathcal{S}_{T(B_0,\psi_0)}^\perp 
	\oplus \mathcal{S}_0 \cap \mathcal{S}_{T(B_0,\psi_0)}$ 
gives an orthogonal decomposition.
Thus \eqref{t-hom2} and \eqref{t-hom1} implies
\begin{align}
f^\dagger(\mathcal{S}_0 \otimes 1_\mathcal{A})f
&\subset f^\dagger((\mathcal{S}_0 \cap \mathcal{S}_{T(B_0,\psi_0)}^\perp 
	\oplus \mathcal{S}_{T(B_0,\psi_0)}) \otimes 1_\mathcal{A})f
\\
&\overset{\eqref{t-hom1}}{\subset}
	 (\mathcal{S}_1 + \mathcal{S}_{T(B_1,\psi_1)}) \otimes \mathcal{A}
\overset{\eqref{t-hom2}}{=} \mathcal{S}_1 \otimes \mathcal{A}.
\end{align}
\end{proof}

\begin{rmk}
For a quantum-to-classical $t$-homomorphism $(f,\mathcal{A}):\mathcal{G}_0 \overset{t}{\to} \mathcal{G}_1=(\C^n,\tau,A_1)$, Proposition \ref{prop:t-homch} (4) is equivalent to the existence of projections $P_1,...,P_n \in B_0 \otimes \mathcal{A}$ satisfying 
$P_i(\mathcal{S}_0\otimes \mathcal{A})P_j=0$ for all $(i,j)$ 
with $\braket{e_i|A_1|e_j}=0$.
Indeed, RHS$-$LHS of (4) with imput $e_i \otimes e_j$ yields
\[
0=n^{-1}
\begin{sd}
\path coordinate (m0*) at (\sep*2/3,\sep/2) 
	  coordinate (m1) at ($-1*(m0*)$)
	  coordinate (m0) at (-\sep/3,\sep*5/6)
	  coordinate (m1*) at ($-1*(m0)$)
	  coordinate (a1) at (-\sep/3,\sep)
	  coordinate (b2) at ($-1*(a1)$) (b2)node[below,draw,circle](ej){$e_j$}
	  coordinate (a2) at (\sep,\sep)
	  coordinate (b1) at ($-1*(a2)$) (b1)node[below,draw,circle](ei){$e_i$}
	  coordinate (ha) at (-\sep,\sep)
	  coordinate (hb) at ($-1*(ha)$);
\draw[name path=rightB] (b1) to[out=90,in=180] (m1) 
	to[out=0,in=180] node[circle,draw,fill=white](A1){$A_1^c$} (m1*)
	to[out=0,in=-90] (m0*);
\draw[name path=leftB] (m1) to[out=90,in=180] (m0) 
	to[out=0,in=180] node[circle,draw,fill=white](A0){$A_0$} (m0*)
	to[out=0,in=90] ([xshift=\sep/4]hb);
\draw (b2)to[out=90,in=-90] (m1*) (m0)to[out=90,in=-90] (a1);
\draw[->-, name path=H] (hb) to[out=90,in=-90] (ha);
\path [name intersections={of= H and rightB, by= f1}];
\path [name intersections={of= H and leftB, by= f2}];
\path node[draw,circle,fill=white] at (f1) {$f$} 
	node[draw,circle,fill=white] at (f2) {$f$};
\end{sd}
=f(e_i)(A_0 \otimes 1_\mathcal{A})f(e_j)
\]
where $A_1^c=J-A_1$ is the complement of $A_1$ satisfying $n\braket{e_i|A_1^c|e_j}=1$. 
We may put $P_i=f(e_i)$ and take Schur product with $\mathcal{S}_0$ from the right to obtain $P_i(\mathcal{S}_0\otimes \mathcal{A})P_j=0$.
Conversely, if we have $P_i$'s, then the desired $f$ is given by $f(e_i)=P_i$.
\end{rmk}

The notion of local homomorphism is stronger than that of graph homomorphism as follows.

\begin{prop}\label{prop:lochom=>hom}
Let $(f,\C):\mathcal{G}_0 \overset{loc}{\to} \mathcal{G}_1$ be a $loc$-homomorphism. Then $f^\op:\mathcal{G}_0 \overset{}{\to} \mathcal{G}_1$ is a graph homomorphism.
\end{prop}

\begin{proof}
Since $A_0 \in \mathcal{S}_0$ and $\C$ is the tensor unit, Proposition \ref{prop:t-homch} (2) with $T=A_0$ shows
$A_1 \bullet (f^\dagger A_0  f )
=f^\dagger A_0  f$.
\end{proof}

The following theorem gives a sufficient condition to make the two notions of homomorphisms coincide.

\begin{thm}\label{thm:hom-lochomequiv}
Let $\mathcal{G}_j = (B_j,\psi_j,A_j,\mathcal{S}_j)$ for $j=0,1$ 
be real quantum graphs with $\delta_j$-forms $\psi_j$ and quantum relations 
$\mathcal{S}_j=\ran (A_j\bullet \cdot)\subset B(L^2(\mathcal{G}_j))$.
Suppose that $f:B_1 \to B_0$ is modular invariant $\sigma_i \circ f=f=f \circ \sigma_i$ and $\mathcal{G}_1$ is Schur central.
Then $f^\op:\mathcal{G}_0 \to \mathcal{G}_1$ is a graph homomorphism
if and only if $(f,\C):\mathcal{G}_0 \overset{}{\to} \mathcal{G}_1$ is a $loc$-homomorphism.
\end{thm}

\begin{proof}
Proposition \ref{prop:lochom=>hom} shows that a local homomorphism is a graph homomorphism. It suffices to show the converse, i.e., 
\begin{align}
\braket{S|f^\dagger Tf}_{\Psi_1}=\delta_1^{2} \bra{1_{B_1}} S^* \bullet (f^\dagger T f ) \ket{1_{B_1}}_{\psi_1}
= 
\begin{sd}
	\path (0,0) node[circle,draw] (S) {$S^*$} 
		(\sep*3/4,0) node[draw] (T) {$f^\dagger T f$};
	\draw (S.north) to[out=90,in=90](T.north) 
		(S.south) to[out=-90,in=-90](T.south) ;
\end{sd}
=0	\label{edgepres}
\end{align}
holds for any $T \in S_0$ and $S \in S_1^\perp$ from the assumption that \eqref{edgepres} holds for $T=A_0$. 

Take $T$ as a normal vector $\braket{T|T}_{\Psi_0}=1$ in $S_0$.

By the following Lemma \ref{lem:Schurcentral}, we may assume that $S \in S_1^\perp$ is a Schur projection because $\mathcal{G}_1$ is Schur central and so is its complement $(B_1, \psi_1, J-A_1, S_1^\perp=\ran (J-A_1))$.

\begin{lem}\label{lem:Schurcentral}
Let $\mathcal{G} = (B,\psi,A,\mathcal{S})$ be a real quantum graph.
Then $\mathcal{G}$ is Schur central if and only if $\mathcal{S}$ is generated by Schur projections.
\end{lem}

And the modular invariance of $f$ enables us to eliminate loops in diagrams as follows:
\begin{align}
\begin{sd}
	\path (0,0) node[circle,draw] (f) {$f^\dagger$};
	\draw (f.south) to[out=-90,in=180]++(-\sep/6,-\loopdiam) 
		to[out=0,in=-90]++(-\sep/6,\loopdiam) --++(0,\sep/2)
		(f.north) --++(0,\sep/4)  ;
\end{sd}
\overset{(f:\text{ real})}{=}
\begin{sd}
	\path (0,0) node[circle,draw] (f) {$f$};
	\draw (f.south) to[out=-90,in=0]++(\sep/6,-\loopdiam) 
		to[out=180,in=-90]++(\sep/6,\loopdiam) --++(0,\sep/2)
		(f.north) --++(0,\sep/4)  ;
\end{sd}
\overset{(f\sigma_{i}=f)}{=}
\begin{sd}
	\path (0,0) node[circle,draw] (f) {$f$};
	\draw (f.north) --++(0,\sep/4)  
		(f.south) to[out=-90,in=-90] ++(\sep/3,0)--++(0,\sep/2)  ;
\end{sd}	;
\\
\begin{sd}
	\path (0,0) node[circle,draw] (f) {$f$};
	\draw (f.north) to[out=90,in=90]++(\loopdiam,\sep/4) 
		to[out=-90,in=-90]++(-\loopdiam,\sep/4)
		(f.south) to[out=-90,in=-90] ++(-\sep/3,0)--++(0,\sep/2)  ;
\end{sd}
\overset{(\sigma_{-i}f=f)}{=}
\begin{sd}
	\path (0,0) node[circle,draw] (f) {$f$};
	\draw (f.north) --++(0,\sep/4)  
		(f.south) to[out=-90,in=-90] ++(-\sep/3,0)--++(0,\sep/2)  ;
\end{sd}
\overset{(f:\text{ real})}{=}
\begin{sd}
	\path (0,0) node[circle,draw] (f) {$f^\dagger$};
	\draw (f.north) --++(0,\sep/4)  
		(f.south) to[out=-90,in=-90] ++(\sep/3,0)--++(0,\sep/2)  ;
\end{sd}. 	\label{modinv:f}
\end{align}
Now we have 
\begin{align}
0 &= \begin{sd}
	\path (0,0) node[circle,draw] (S) {$S^*$} 
		(\sep*3/4,0) node[draw] (T) {$f^\dagger A_0 f$};
	\draw (S.north) to[out=90,in=90](T.north) 
		(S.south) to[out=-90,in=-90](T.south) ;
\end{sd}
 \overset{(S^*=S=S\bullet S)}{=}
 \delta_1^{-2}
\begin{sd}
	\path (0,0) node[circle,draw] (S1) {$S$} 
		(-\sep*0.6,0) node[circle,draw] (S2) {$S^*$} 
		(\sep*3/4,0) node[draw] (T) {$f^\dagger A_0 f$};
	\draw (S1.north) to[out=90,in=90]coordinate[midway] (m) (S2.north) 
		(S1.south) to[out=-90,in=-90]coordinate[midway] (m*) (S2.south)
		(m) to[out=90,in=90](T.north) 
		(m*) to[out=-90,in=-90](T.south) ;
\end{sd}
 =  \delta_1^{-2}
\begin{sd}
	\path (0,0) node[circle,draw] (S1) {$S$} 
		(-\sep*3/5,0) node[circle,draw] (S2) {$S^*$} 
		(\sep*3/4,0) node[draw] (T) {$f^\dagger A_0 f$}
		(T)++(\sep*0.3,\sep*0.6) coordinate (B2a) 
		(T)++(-\sep*0.3,\sep*0.6) coordinate (B1a) 
		(T)++(\sep*0.3,-\sep*0.6) coordinate (B2b) 
		(T)++(-\sep*0.3,-\sep*0.6) coordinate (B1b);
	\draw (B1b) to[out=90,in=90]coordinate[midway] (m) (B2b) 
		(B1a) to[out=-90,in=-90]coordinate[midway] (m*) (B2a)
		(m) to[out=90,in=-90](T.south) 
		(m*) to[out=-90,in=90](T.north) 
		(B1a) to[out=90,in=90](S1.north) 
		(B1b) to[out=-90,in=-90](S1.south) 
		(B2a) to[out=90,in=90](S2.north) 
		(B2b) to[out=-90,in=-90](S2.south) ;
\end{sd}
\\
& \overset{(f: \text{hom})}{=} \delta_1^{-2}
\begin{sd}
	\path (0,0) node[circle,draw] (S1) {$S$} 
		(-\sep*3/5,0) node[circle,draw] (S2) {$S^*$} 
		(\sep*3/4,0) node[circle,draw] (T) {$A_0$}
		(T)++(\sep*0.3,\sep*3/4) node[circle,draw] (B2a) {$f^\dagger$}
		(T)++(-\sep*0.3,\sep*3/4) node[circle,draw] (B1a) {$f^\dagger$} 
		(T)++(\sep*0.3,-\sep*3/4) node[circle,draw] (B2b) {$f$} 
		(T)++(-\sep*0.3,-\sep*3/4) node[circle,draw] (B1b) {$f$};
	\draw (B1b) to[out=90,in=90]coordinate[midway] (m) (B2b) 
		(B1a) to[out=-90,in=-90]coordinate[midway] (m*) (B2a)
		(m) to[out=90,in=-90](T.south) 
		(m*) to[out=-90,in=90](T.north) 
		(B1a) to[out=110,in=90](S1.north) 
		(B1b) to[out=-110,in=-90](S1.south) 
		(B2a) to[out=90,in=90](S2.north) 
		(B2b) to[out=-90,in=-90](S2.south) ;
\end{sd}
= \delta_1^{-2}
\begin{sd}
	\path (-\sep*3/4,0) node[circle,draw] (S1) {$S$} 
		(\sep*3/4,0) node[circle,draw] (S2) {$S^*$} 
		(0,0) node[circle,draw] (T) {$A_0$}
		(T)++(\sep*0.3,\sep*3/4) node[circle,draw] (B2a) {$f^\dagger$}
		(T)++(-\sep*0.3,\sep*3/4) node[circle,draw] (B1a) {$f^\dagger$} 
		(T)++(\sep*0.3,-\sep*3/4) node[circle,draw] (B2b) {$f$} 
		(T)++(-\sep*0.3,-\sep*3/4) node[circle,draw] (B1b) {$f$};
	\draw (B1b) to[out=90,in=90]coordinate[midway] (m) (B2b) 
		(B1a) to[out=-90,in=-90]coordinate[midway] (m*) (B2a)
		(m) to[out=90,in=-90](T.south) 
		(m*) to[out=-90,in=90](T.north) 
		(B1a) to[out=110,in=90](S1.north) 
		(B1b) to[out=-110,in=-90](S1.south) 
		(B2a) to[out=50,in=0] ++(\sep*0.4,\sep/2)
		to[out=180,in=90](S2.north) 
		(B2b) to[out=-50,in=0] ++(\sep*0.4,-\sep/2)
		to[out=180,in=-90](S2.south) ;
\end{sd}
\overset{(f: \text{real})}{=} \delta_1^{-2}
\begin{sd}
	\path (-\sep*3/4,0) node[draw] (S1) {$fSf^\dagger$} 
		(\sep*3/4,0) node[draw] (S2) {$fS^*f^\dagger$} 
		(0,0) node[circle,draw] (T) {$A_0$}
		(T)++(\sep*0.3,\sep*1/2) coordinate (B2a) 
		(T)++(-\sep*0.3,\sep*1/2) coordinate (B1a) 
		(T)++(\sep*0.3,-\sep*1/2) coordinate (B2b) 
		(T)++(-\sep*0.3,-\sep*1/2) coordinate (B1b) ;
	\draw (B1b) to[out=90,in=90]coordinate[midway] (m) (B2b) 
		(B1a) to[out=-90,in=-90]coordinate[midway] (m*) (B2a)
		(m) to[out=90,in=-90](T.south) 
		(m*) to[out=-90,in=90](T.north) 
		(B1a) to[out=110,in=90](S1.north) 
		(B1b) to[out=-110,in=-90](S1.south) 
		(B2a) to[out=90,in=0] ++(\sep/4,\sep/2)
		to[out=180,in=90](S2.north) 
		(B2b) to[out=-90,in=0] ++(\sep/4,-\sep/2)
		to[out=180,in=-90](S2.south) ;
\end{sd}
\\
& = \delta_1^{-2}
\begin{sd}
	\path (0,0) node[circle,draw] (T) {$A_0$}
		(T)++(\sep*0.4,\sep*0.5) coordinate (B2a) 
		(T)++(-\sep*0.5,\sep*0.1) coordinate (B1a) 
		(T)++(\sep*0.5,-\sep*0.1) coordinate (B2b) 
		(T)++(-\sep*0.4,-\sep*0.5) coordinate (B1b) 
		(B1a)++ (0,\sep*2/3) node[draw] (S1) {$fSf^\dagger$} 
		(B2b)++ (0,-\sep*2/3) node[draw] (S2) {$fS^* f^{\dagger}$} ;
	\draw (B2b)%--++(0,\sep*2/3)
		 to[out=90,in=90]coordinate[midway] (m) (T.north) 
		(B1a)%--++(0,-\sep*2/3) 
		 to[out=-90,in=-90]coordinate[midway] (m*) (T.south)
		(m) to[out=90,in=-90](B2a) 
		(m*) to[out=-90,in=90](B1b) 
		(B1a) to[out=90,in=-90](S1.south) 
		(B2a) to[out=90,in=90](S1.north) 
		(B1b) to[out=-90,in=-90] ++(-\loopdiam,-\sep/4)
		 to[out=90,in=90] ++(\loopdiam,-\sep/4) 
		 to[out=-90,in=-90](S2.south) 
		(B2b) to[out=-90,in=-90] ($(B2b)!0.5!(S2.north)+(\loopdiam,0)$) 
		 to[out=90,in=90](S2.north) ;
\end{sd}
 \overset{\eqref{modinv:f}}{=} \delta_1^{-2}
\begin{sd}
	\path (0,0) node[circle,draw] (T) {$A_0$}
		(T)++(\sep*0.4,\sep*0.5) coordinate (B2a) 
		(T)++(-\sep*0.5,\sep*0.1) coordinate (B1a) 
		(T)++(\sep*0.5,-\sep*0.1) coordinate (B2b) 
		(T)++(-\sep*0.4,-\sep*0.5) coordinate (B1b) 
		(B1a)++ (0,\sep*2/3) node[draw] (S1) {$fSf^\dagger$} 
		(B1b) node[draw,below] (S2) {$fS^\dagger f^{\dagger}$} ;
	\draw (B2b)%--++(0,\sep*2/3)
		 to[out=90,in=90]coordinate[midway] (m) (T.north) 
		(B1a)%--++(0,-\sep*2/3) 
		 to[out=-90,in=-90]coordinate[midway] (m*) (T.south)
		(m) to[out=90,in=-90](B2a) 
		(m*) to[out=-90,in=90](B1b) 
		(B1a) to[out=90,in=-90](S1.south) 
		(B2a) to[out=90,in=90](S1.north) 
		(B2b) to[out=-90,in=-45](S2) ;
\end{sd}
\overset{\eqref{graph-rel}}{=} \delta_0^2 \delta_1^{-2}
\begin{sd}
	\path (0,0) node[block,draw] (T) {$P_{A_0}$}
		(T)++(\sep*0.4,\sep*0.25) coordinate (B2a) 
		(T)++(-\sep*0.4,\sep*0.25) coordinate (B1a) 
		(T)++(\sep*0.4,-\sep*0.25) coordinate (B2b) 
		(T)++(-\sep*0.4,-\sep*0.25) coordinate (B1b) 
		(B1a)++ (0,\sep*0.4) node[draw] (S1) {$fSf^\dagger$} 
		(B1b)++ (0,-\sep*0.4) node[draw] (S2) {$fS^\dagger f^{\dagger}$} ;
	\draw (B1a) --(S1) (B1b) --(S2) 
		(B2a) to[out=90,in=40](S1) 
		(B2b) to[out=-90,in=-40](S2) ;
\end{sd},
\end{align}
where the non-indicated equalities are continuous deformations.

Recall \eqref{graph-rel} that 
$P_{A_0}=\delta_0^{-2} \begin{sd}
	\path (0,0) node[circle,draw] (A) {$A_0$};
	\draw (A.north)
		 to[out=90,in=90]coordinate[midway] (m) 
		([xshift=\sep/3]A.north) --++(0,-\sep/2)
		(A.south) to[out=-90,in=-90]coordinate[midway] (m*) 
		([xshift=-\sep/3]A.south) --++(0,\sep/2)
		(m) --++ (0,\sep/8)
		(m*) --++(0,-\sep/8);
\end{sd}$ is a projection onto $\iota(\mathcal{S}_0)\subset B_0 \otimes B_0$. Since $\iota(T) \in \iota(\mathcal{S}_0)$, the rank one projection 
$\ket{\iota(T)} \bra{\iota(T)}=\begin{sd}
	\path (0,\sep*0.275)node[circle,draw] (Ta) {$T$}
		(0,-\sep*0.275)node[circle,draw] (Tb) {$T^\dagger$} ;
	\draw (Ta)--++(0,\sep/3) (Ta)to[out=-150,in=-90]++(-\sep/2,\sep/3)
		 (Tb)--++(0,-\sep/3) (Tb)to[out=150,in=90]++(-\sep/2,-\sep/3);
\end{sd}$
 is smaller than or equal to $P_{A_0}$,
hence $P=P_{A_0}-\ket{\iota(T)} \bra{\iota(T)}$ is also a projection.
Therefore we obtain
\begin{align}
0=\delta_0^2 \delta_1^{-2} \left(
\begin{sd}
	\path (0,\sep*0.5)node[circle,draw] (Ta) {$T$}
		(0,-\sep*0.5)node[circle,draw] (Tb) {$T^\dagger$}
		(-\sep*0.8,\sep*0.5) node[draw] (Sa) {$fSf^\dagger$} 
		(-\sep*0.8,-\sep*0.5) node[draw] (Sb) {$fS^\dagger f^{\dagger}$} ;
	\draw (Ta) to[out=90,in=80](Sa)
		(Ta)to[out=-90,in=-80](Sa)
		(Tb)to[out=-90,in=-80](Sb) 
		(Tb)to[out=90,in=80](Sb);
\end{sd}
+
\begin{sd}
	\path (0,0) node[block,draw] (T) {$P$}
		(T)++(\sep*0.4,\sep*0.25) coordinate (B2a) 
		(T)++(-\sep*0.4,\sep*0.25) coordinate (B1a) 
		(T)++(\sep*0.4,-\sep*0.25) coordinate (B2b) 
		(T)++(-\sep*0.4,-\sep*0.25) coordinate (B1b) 
		(B1a)++ (0,\sep*0.4) node[draw] (S1) {$fSf^\dagger$} 
		(B1b)++ (0,-\sep*0.4) node[draw] (S2) {$fS^\dagger f^{\dagger}$} ;
	\draw (B1a) --(S1) (B1b) --(S2) 
		(B2a) to[out=90,in=40](S1) 
		(B2b) to[out=-90,in=-40](S2) ;
\end{sd}
\right).
\end{align}
By the vertical symmetry, each term is nonnegative, hence they must be zero. 
The first term is what we desired:
\begin{align}
0=\left \vert
\begin{sd}
	\path (0,\sep*0.5)node[circle,draw] (Ta) {$T$}
		(-\sep*0.8,\sep*0.5) node[draw] (Sa) {$fSf^\dagger$}  ;
	\draw (Ta) to[out=90,in=80](Sa)
		(Ta)to[out=-90,in=-80](Sa);
\end{sd}
\right \vert^2
 \overset{(f:\text{real})}{=}
\left \vert
\begin{sd}
	\path (0,\sep*0.5)node[draw] (Ta) {$f^\dagger T f$}
		(-\sep*0.8,\sep*0.5) node[circle,draw] (Sa) {$S$}  ;
	\draw (Ta) to[out=90,in=80](Sa)
		(Ta)to[out=-90,in=-80](Sa);
\end{sd}
\right \vert^2
\overset{(S:\text{real})}{=}
\abs{ \braket{S|f^\dagger T f} }^2
.
\end{align}

\end{proof}

\begin{thm}\label{thm:hom-lochomequiv_tracial}
Let $\mathcal{G}_j = (B_j,\psi_j,A_j,\mathcal{S}_j)$ for $j=0,1$ 
be real tracial quantum graphs such that $\mathcal{G}_1$ is Schur central.
Then $f^\op:\mathcal{G}_0 \to \mathcal{G}_1$ is a graph homomorphism
if and only if $(f,\C):\mathcal{G}_0 \overset{}{\to} \mathcal{G}_1$ is a $loc$-homomorphism.
\end{thm}

\begin{proof}
Since each $\psi_j=\Tr(Q_j \cdot)$ is tracial, the density $Q_j$ is central and its modular automorphism is 
$\sigma_i=Q_j^{-1} (\cdot) Q_j=\id_{B_j}$. Thus we have 
$\sigma_i \circ f=f=f \circ \sigma_i$. 
Therefore the statement follows from Theorem \ref{thm:hom-lochomequiv}.
\end{proof}

\begin{proof}[{Proof of Lemma \ref{lem:Schurcentral}}]
Since the statement depends only on the Schur product structure, it suffices to show for a von Neumann algebra $\mathcal{M}(=B^\op \otimes B)$ and a projection $p \in \mathcal{M}$ that $p$ is central if and only if $p\mathcal{M}$ is linearly generated by projections in weak operator topology (WOT). 

Suppose $p$ is central, then $p\mathcal{M}=p\mathcal{M}p$ is a WOT-closed subalgebra of $\mathcal{M}$. 
Then we can decompose $x \in p\mathcal{M}p$ into real and imaginary parts,
 which have spectral projections in $p\mathcal{M}p$. Since $x$ lies in the WOT-closed linear span of such spectral projections, we are done.

Suppose that $p\mathcal{M}$ is generated by projections. 
It follows for any projection $q\in p\mathcal{M}$ that 
$pq=q=q^*=qp$. 
Since such projections $q$ span $p\mathcal{M}$ in WOT, we have $px=pxp$ for any $x \in \mathcal{M}$, and $px^*=px^*p$ as well. Thus we get $px=xp$, i.e., $p$ is central.
\end{proof}

\subsection{$t$-$2$ colorability compared with bipartiteness}

\begin{dfn}[{\cite{BGH2022quantum}}]
Let $t \in \{loc, q, qa, qc, C^*, alg\}$ and $c \in \Z_{>0}$.
A quantum graph $\mathcal{G}$ is $t$-$c$ colorable if 
there exists a $t$-homomorphism $\mathcal{G} \to K_c$, 
which is called a $t$-$c$ coloring of $\mathcal{G}$.
The $t$-chromatic number of $\mathcal{G}$ is defined by 
$\chi_t(\mathcal{G})=\inf\{c \in \Z_{>0} \vert \mathcal{G}: t \hyphen c\textrm{ colorable}\}$.
\end{dfn}

Note that a $t$-$c$ coloring need not be a surjective $t$-homomorphism. Surjectivity means that it uses all the $c$ colors.

\begin{rmk}
By the obvious inclusion of the classes of algebras, $c$-colorability has the following implication:
$loc \Rightarrow q \Rightarrow qa \Rightarrow qc \Rightarrow C^* \Rightarrow alg,$ 
and hence the chromatic numbers satisfy
\[
\chi_{loc} \geq \chi_q \geq  \chi_{qa} \geq \chi_{qc} \geq \chi_{C^*} \geq \chi_{alg}.
\]
\end{rmk}

\begin{prop} \label{prop:2col=>sym}
Let $\mathcal{G}=(B,\psi,A)$ be an $alg$-$2$ colorable real quantum graph.
Then $\mathcal{G}$ has a symmetric spectrum $\Spec A=-\Spec A$. Moreover, if it is $q$-$2$ colorable, then the symmetry of the spectrum holds with its multiplicity.
\end{prop}

\begin{proof}
If $A=0$, the statement is trivial. So we may assume $A \neq 0$.
Let $(f,\mathcal{A}) : \mathcal{G} \to K_2
=(\C^2, \tau, A_{K_2},\mathcal{S}_{K_2})$
 be an $alg$-homomorphism. 
In this case $(f,\mathcal{A})$ is automatically surjective, i.e., 
$f:\C^2 \to B\otimes \mathcal{A}$ is injective. 
Indeed if $f$ is not injective, then we may assume $f(e_1)=1_B \otimes 1_\mathcal{A}$ and $f(e_2)=0$ without loss of generality. 
But this implies for $e_1 e_1^\dagger \in \mathcal{S}_{K_2}^\perp$ that
$\braket{e_1 e_1^\dagger | f^\dagger (A \otimes 1_\mathcal{A}) f}_{\Tr}
=\braket{e_1| f^\dagger (A \otimes 1_\mathcal{A}) f|e_1}_{\tau}
=\braket{1| A |1}_{\psi} 1_\mathcal{A} \neq 0$ 
by Lemma \ref{lem:numofedges}, 
which contradicts that $(f,\mathcal{A})$ is an $alg$-homomorphism
(Proposition \ref{prop:t-homch}~(3)).

Now we have nonzero projections 
$P_j=\lambda(f(e_j)) \in B(L^2(\mathcal{G})) \otimes \mathcal{A}$, 
where $\lambda$ denotes the left multiplication, satisfying
$P_j(A\otimes 1_\mathcal{A})P_j=0$ for each $j=1,2$.
Then we have
\begin{align}
A\otimes 1_\mathcal{A}
=P_1(A\otimes 1_\mathcal{A})P_2 + P_2(A\otimes 1_\mathcal{A})P_1.
\end{align}
and hence
\begin{align}
(A\otimes 1_\mathcal{A})(P_1-P_2)
&= P_2(A\otimes 1_\mathcal{A})P_1 - P_1(A\otimes 1_\mathcal{A})P_2
\\
&= (P_2-P_1)(A\otimes 1_\mathcal{A}),
\\
((\alpha \,\id+A)\otimes 1_\mathcal{A})(P_1-P_2)
&=(P_1-P_2)((\alpha \,\id - A)\otimes 1_\mathcal{A})
\end{align}
It follows for $\alpha \in \Spec A$ and $v \in \ker(\alpha \,\id_B -A)$ that
$(P_1-P_2)v \in \ker(\alpha \,\id_B +A) \otimes \mathcal{A}$.
Indeed $v$ satisfies
\begin{align}
((\alpha \,\id+A)\otimes 1_\mathcal{A})(P_1-P_2)v
&= (P_1-P_2)((\alpha \,\id - A)\otimes 1_\mathcal{A})v=0.
\end{align}
For a generalized eigenvector $v \in \ker(\alpha \,\id_B -A)^k$ for some positive integer $k$, we similarly have
$(P_1-P_2)v \in \ker(\alpha \,\id_B +A)^k \otimes \mathcal{A}$ by
\begin{align}
((\alpha \,\id+A)^k\otimes 1_\mathcal{A})(P_1-P_2)v
&= (P_1-P_2)((\alpha \,\id - A)^k\otimes 1_\mathcal{A})v=0.
\end{align}
Therefore $-\alpha \in \Spec A$, i.e., $\mathcal{G}$ has a symmetric spectrum.

If $(f,\mathcal{A})$ is a $q$-$2$ coloring, then $\mathcal{A}\subset M_n=B(\C^n)$ for some positive integer $n$. Thus $P_1-P_2$ restricts to linear isomorphisms between generalized eigenspaces 
$\ker(\alpha \,\id_B -A)^{\dim B} \otimes \C^n 
\cong \ker(\alpha \,\id_B +A)^{\dim B} \otimes \C^n$,
hence the multiplicities coincide as $\dim \ker(\alpha \,\id_B -A)^{\dim B}  
=\dim \ker(\alpha \,\id_B +A)^{\dim B}$.
\end{proof}

\begin{thm}\label{thm:bipartiff2col}
Let $\mathcal{G}=(B,\psi,A)$ be a real tracial quantum graph. Then $\mathcal{G}$ is bipartite if and only if it is $loc$-$2$ colorable.
\end{thm}

\begin{proof}
By Theorem \ref{thm:hom-lochomequiv_tracial}, the existence of a graph homomorphism $\mathcal{G} \to K_2$ is equivalent to the existence of a $loc$-homomorphism $\mathcal{G} \to K_2$ because these are tracial real quantum graphs and the classical $K_2$ is Schur central.
Thus $\mathcal{G}$ is bipartite if and only if it is $loc$-$2$ colorable.
\end{proof}

\begin{cor} \label{cor:2colequiv}
Let $\mathcal{G}=(B,\psi,A)$ be a connected $d$-regular undirected tracial quantum graph.
The following are equivalent:
\begin{description}
\item[(1)] $\mathcal{G}$ is $loc$-$2$ colorable;
\item[(2)] $\mathcal{G}$ is $alg$-$2$ colorable;
\item[(3)] $\mathcal{G}$ has a symmetric spectrum;
\item[(4)] $-d \in \Spec(A)$.
If $d=0$, we require ${\dim B} \geq 2$;
\item[(5)] $\mathcal{G}$ is bipartite.
\end{description}
In this case, the symmetry of the spectrum in (3) holds with multiplicity.
\end{cor}
\begin{proof}
\noindent$((1)\implies(2))$: Obvious by definition.

\noindent$((2)\implies(3))$: The symmetry follows from Proposition \ref{prop:2col=>sym}. 
In particular, if we assume (1), the symmetry holds with multiplicity.

\noindent$((3)\implies(4))$: Since $\mathcal{G}$ is $d$-regular, the symmetry of spectrum shows $-d \in \Spec A$.

\noindent$((4)\implies(5))$: This is shown by Theorem \ref{thm:connbipartite} 
as we assumed that $\mathcal{G}$ is connected.

\noindent$((5)\implies(1))$:
This is the direct consequence of Theorem \ref{thm:bipartiff2col}.
\end{proof}

In particular, this means that all kinds of $t$-$2$ colorability are mutually equivalent for connected regular undirected tracial quantum graphs.

\bibliographystyle{plain}
\bibliography{bunken.bib}

\end{document}